\documentclass[11pt]{article}
\usepackage{amsfonts}
\usepackage{amscd}
\usepackage{amsmath,amssymb,amsthm}
\usepackage{amstext}
\usepackage{amscd}
\usepackage{graphicx}
\usepackage[all]{xy}

\theoremstyle{plain}
\newtheorem{lem}{Lemma}[section]

\newtheorem*{embedded}{Theorem A}

\newtheorem*{main1}{Theorem 1}
\newtheorem*{main2}{Theorem 2}

\newtheorem{cla}[lem]{Claim}
\newtheorem{coro}[lem]{Corollary}
\newtheorem{prop}[lem]{Proposition}

\theoremstyle{definition}

\newtheorem{definition}[lem]{Definition}
\newtheorem{rems}[lem]{Remarks}
\newtheorem{rem}[lem]{Remark}

\renewcommand{\descriptionlabel}[1]%
       {\hspace{\labelsep}\textsf{#1}}

\DeclareMathOperator{\dem}{dem}

\DeclareMathOperator{\Int}{Int}

\begin{document}
\title{Wild knots embedded in the Menger sponge \,\,\thanks{ 
  A.V. would like to acknowledge DGAPA, Universidad Nacional Aut\'onoma de
M\'exico. Proyecto PAPIIT IN103324, for the financial support.}}

\author{Gabriela Hinojosa, Ulises Morales-Fuentes, \\ Rogelio Valdez, Alberto Verjovsky}
\date{August 31, 2026}

\maketitle

\begin{abstract} 
In this paper, we provide explicit recursive constructions of infinitely many non-equivalent wild knots contained in the Menger sponge, in such a way that we can control their set of wild points, which lies in a usual Cantor set contained in the Menger sponge. Furthermore, we show that wild knots of dynamically defined type arising from Kleinian group actions can be isotoped into the sponge.  We want to emphasize that our approach is constructive and geometric. 
\end{abstract}

\section{Introduction}

In 1926, Karl Menger introduced his fractal sponge, known as the Menger sponge or Menger cube. This object generalizes the one-dimensional Sierpi\'ns\-ki carpet in the unit square and the one-dimensional Cantor set. The Menger sponge is constructed iteratively. Start with a unit cube $I^3$ and divide it into 27 smaller cubes of side length $\frac{1}{3}$. Remove the interior of the central cube (the one at the center) and also remove the interior of the cubes at the centers of each face. This step leaves 20 smaller cubes for the first stage (level). Next, repeat the process on each of the 20 remaining cubes to reach the second stage. Continue iterating this process. The Menger sponge $M$ is the inverse limit of this construction. Menger showed that the sponge is universal for all compact one-dimensional topological spaces. In particular, every curve is homeomorphic to a subset of the Menger sponge.\\



\noindent In 2024, three high schoolers --Niko Voth, Joshua Broden, and Noah Na\-za\-reth-- and their mentor Malors Espinosa proved again, but in a constructive way, that 
any tame knot $K$ can be embedded into a finite iteration of the Menger sponge (\cite{BENV}). \\

\noindent In this paper, we extend the above to wild knots. We provide explicit recursive constructions of infinitely many non-equivalent wild knots contained in the Menger sponge, in such a way that we can control their set of wild points that lies in a usual Cantor set contained in the Menger sponge. We aim to prove:

\begin{main1}\label{main1}
For every finite subset $T=\{p_1,\ldots,p_n\}$ of the standard Cantor set $\mathcal{C}$ contained in the edge $e\subset M$ described in Section~3.2, there is an explicit, recursively constructed wild knot $\mathcal{K}_T\subset M$ whose set of wild points is exactly $T$. Consequently, the Menger sponge $M$ contains infinitely many pairwise non-equivalent, explicitly constructed wild knots, distinguished by the cardinality of their wild point sets.
\end{main1}

\noindent We were unaware that non-recursive general existence results for these subjects had previously been established in Russia. In 1971, M. A. Shtan'ko resolved Menger's universality problem for compacta in Euclidean space (\cite{shtan'ko}, \cite{daverman}). Specifically, Shtan'ko demonstrated that the Menger sponge is ambient isotopically universal for 1-dimensional compacta in Euclidean 3-space such that its demension is also one\footnote{\emph{Demension} (embedding dimension) is specialized terminology from decomposition-space theory, introduced by R.\,D. Edwards \cite{edwards} (see also Daverman--Venema \cite{daverman}). It is not in wide use outside this corner of geometric topology, so we recall its meaning for the case needed here -- compacta of dimension at most one, the only case used in this paper. A compactum $X\subset\mathbb{S}^3$ has embedding dimension one, $\dem(X)=1$, if $X=\bigcap_k H_k$ for some nested sequence of compact handlebodies $H_k\subset\mathbb{S}^3$ such that, for every $\varepsilon>0$, all sufficiently deep $H_k$ admit an $\varepsilon$-retraction onto a finite graph inside $H_k$; this is exactly the refined notion of tameness for which Shtan'ko's universality theorem is proved, and it is a priori different from the ordinary topological (covering) dimension $\dim$, which is why we keep the two symbols distinct throughout the paper. Every compactum of topological dimension $\leq1$ that arises in this paper -- in particular, every knot, tame or wild -- has demension one as well, by the approximation theorem mentioned in Section 2.1 below, which is why Theorem A and Theorems 1--2 may legitimately invoke Shtan'ko's theorem.} (\cite{shtan'ko}, \cite{daverman}).
As a consequence of this theorem, any knot (tame or wild) embeds isotopically into the Menger sponge, since a knot, being a topological circle, has demension one. We want to underscore that the proof of this theorem shows the existence of such an isotopy. However, there are examples of compact one-dimensional spaces that are not isotopic to a subset of the Menger sponge; for more details, see \cite{mcmillan}.\\

\begin{rem}[Relation to Shtan'ko's theorem]\label{shtankorelation}
The bare existence statement ``$M$ contains infinitely many pairwise non-equivalent wild knots'' already follows non-constructively from Shtan'ko's theorem together with the classical fact that $\mathbb{S}^3$ contains, for every $n\in\mathbb{N}$, a wild knot with exactly $n$ wild points (cf.\ the Fox--Artin type constructions): each such knot has demension one, hence embeds ambient-isotopically in $M$ by Shtan'ko's theorem, and knots with different numbers of wild points are pairwise non-equivalent. Shtan'ko's theorem, however, gives no information about \emph{which} isotopic copy of a given wild knot lands in $M$, no formula for the embedding, and no control on which points of $M$ become the wild points of the image. The content of Theorem 1 in the sharper, constructive form stated above -- and of Theorem 2 -- is precisely this missing information: an explicit recursive construction, compatible stage by stage with the construction of $M$ itself (Section~3.2), that produces a wild knot inside $M$ whose wild point set is \emph{exactly} a prescribed finite subset of a fixed tame Cantor set in $M$, together with an explicit isotopy embedding any dynamically defined wild knot into $M$ with the same control. This explicit, geometric control -- not the bare existence statement -- is the paper's contribution beyond Shtan'ko's theorem.
\end{rem}

\noindent Shtan'ko's work introduced concepts that later became foundational in geometric topology and decomposition theory. Notably, certain isotopy constructions in Shtan'ko's paper anticipated techniques subsequently developed by R. D. Edwards in the study of cell-like maps and the cell-like approximation theorem (\cite{edwards}).\\

\noindent  We emphasize that the approach presented here is constructive and geometric. The Menger sponge serves as a fractal environment that supports recursive knot operations at infinitely many scales, enabling the construction of wild knots via iterated connected sums embedded within the sponge. Because these results are constructive and concrete, it is possible to create simulations that visualize the knots, which is a significant advantage of contemporary computational tools.  \\

\noindent Furthermore, we show that wild knots of dynamically defined type arising from Kleinian group actions can be isotoped into the Menger sponge. More specifically, J. P. D\'iaz and G. Hinojosa in \cite{DH} constructed {\it wild knots of dynamically defined type} as follows. Consider a tame knot $K$. An $n$-beaded necklace $T^{\circ}$ subordinate to $K$ consists of the union of $n$ disjoint closed round 3-balls, called pearls, $B_j$ ($j=1,\ldots, n$) whose centers lie on $K$, and such that the segment of $K$ contained in each pearl is unknotted. An $n$-pearl chain necklace $T$ is the union of $T^{\circ}$ and $K$; \emph{i.e.} $T=T^{\circ}\cup K$. Let $\Gamma_{T^{\circ}}$ be the group generated by reflections $I_{j}$ through the boundary of each pearl $\Sigma_{j}=\partial B_j$ ($j=1,\ldots,n$). Then the action of $\Gamma_{T^{\circ}}$ on $T$ yields a sequence of nested pearl chain necklaces $T_k$, whose inverse limit is the desired wild knot of dynamically defined type $\Lambda(K,T^{\circ})$. Notice that $\Lambda(K,T^{\circ})$ is wild at a Cantor set of points.\\

\noindent In this context, we have the following result.

\begin{main2}\label{main2}
Any wild knot of dynamically defined type is isotopic to a wild knot contained in the Menger sponge such that its set of wild points is contained in a tame Cantor set contained in $M$.
\end{main2}

\noindent In the last section, we consider a different, two-dimensional cubical fractal continuum containing $M$ (not to be confused with the classical, one-dimensional Sierpi\'nski carpet of the unit square), and we obtain analogous results.\\

\section{Preliminaries}

In this section, we briefly recall some definitions of knots; for more details, see \cite{rolfsen}, \cite{DH}.\\

\subsection{The Menger sponge and its universality}

\noindent The Menger sponge is a three-dimensional generalization of the one-di\-men\-sio\-nal Cantor set and the one-dimensional Sierpi\'nski carpet in the unit square. Its construction is also a recursive process (see Figure \ref{Menger}).\\

\begin{enumerate}
\item {\it First stage}. We start with the unit cube $I^3$ and subdivide each face into nine squares (as in the Sierpi\'nski carpet construction). This operation subdivides the cube into 27 subcubes. Now we remove the interior of the subcube at the center of each face and also the interior of the central subcube of $I^3$ to obtain $M_1$, which consists of the union of 20 subcubes of side length $\frac{1}{3}$.\\

\item {\it Second stage}. We repeat the operation on each of the 20 remaining cubes of the first stage, subdividing each into 27 subcubes, and again remove the interior of the subcube at the center of each face and the interior of the central subcube of each cube. This yields $M_2$, consisting of the union of $20^2$ subcubes of side length $\frac{1}{3^2}$.\\

\item {\it $k^{th}$ stage}. We repeat the same process on each cube of the $(k-1)$-th stage to obtain $M_k$, which consists of the union of $20^{k}$ subcubes of side length $\frac{1}{3^k}$. 
\end{enumerate}

\begin{figure}[h] 
\begin{center}
\includegraphics[height=4.7cm]{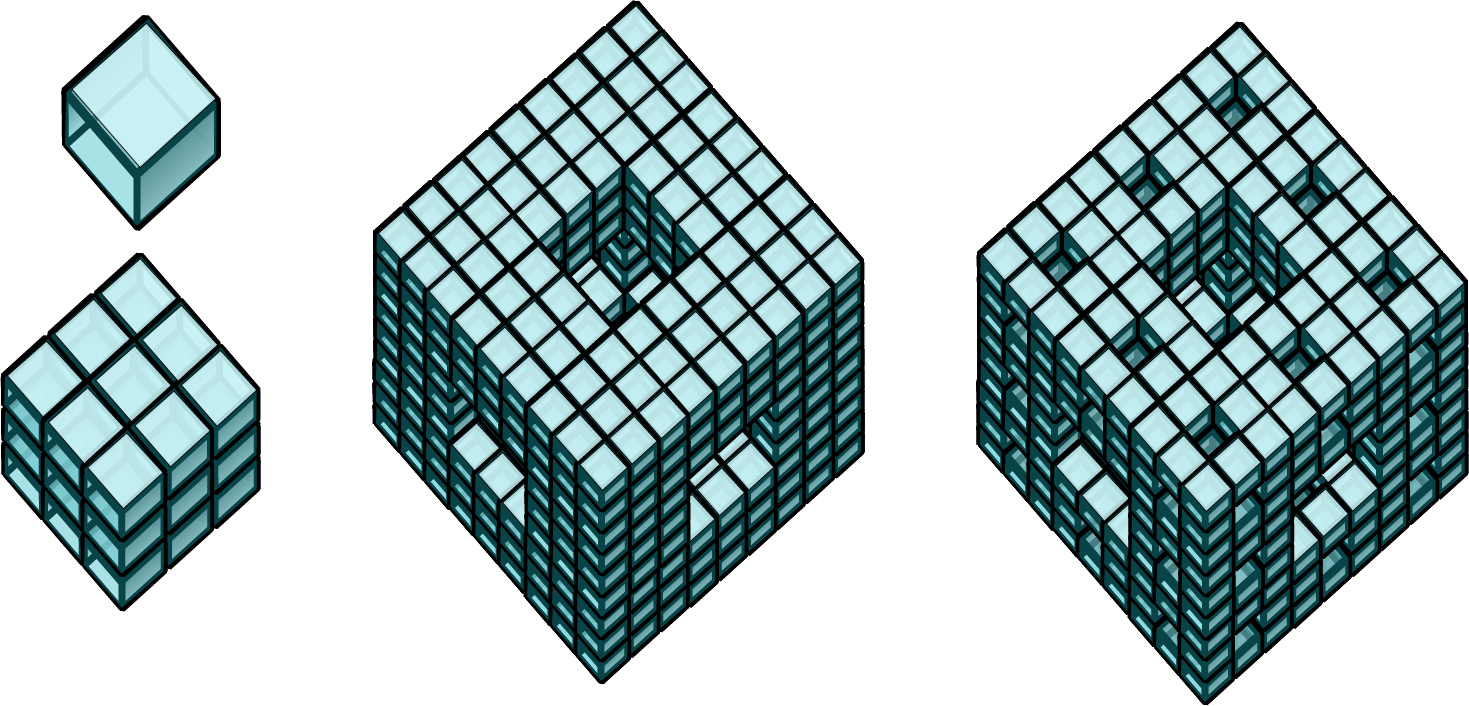}
\end{center}
\caption{\sl First steps in the construction of the Menger sponge.} 
\label{Menger}
\end{figure} 

\noindent Then the \emph{Menger sponge} is defined as the inverse topological limit
$$
M=\varprojlim_{k} M_{k}=\bigcap_{k=1}^{\infty} M_{k}.
$$

\noindent By construction, each face of $M$ is a Sierpi\'nski carpet whose area goes to zero. The sponge's Hausdorff dimension is $\frac{\log 20}{\log 3}\sim 2.727$.\\

\noindent In \cite{BENV}, the following result was proved, which will be very useful for our purpose.

\begin{lem}\label{dust-points}
If $(x,y)$ is a point in the Cantor dust, then for all $0\leq z\leq 1$ the point $(x,y,z)$ belongs to the Menger sponge $M$.
\end{lem}

\noindent The Menger sponge is one of the most remarkable fractal objects in topology. Menger showed that it is a universal one-dimensional compactum, \emph{i.e.}, it contains homeomorphic copies of all compact metric spaces of dimension one. Later, in 1971, M. A. Shtan'ko proved that the Menger sponge is ambient isotopy universal for 1-dimensional compacta in Euclidean 3-space (\cite{shtan'ko}, \cite{daverman}). \\

\noindent {\bf Theorem} (Shtan'ko). \emph{Let $K\subset\mathbb{R}^3$ be compact with $\dim K=r\leq 1$. Then there exists an ambient isotopy $F_t:\mathbb{R}^3\rightarrow\mathbb{R}^3$ such that $F_1 (K)\subset M$ if and only if $\dem (K)=r$, where $\dem (K)$ denotes the embedding dimension (demension) of $K$.}\\

\noindent The proof of this theorem splits into two parts. The first is to show, using a geometric isotopy argument, that any compact of demension $1$ can be pushed into a modified Menger compactum. The second one is to use a reduction theorem to show that the modified Menger compactum is ambient isotopic to the standard Menger compactum. However, to apply the theorem to arbitrary compacta of dimension 1, one needs an additional ingredient: an approximation theorem ensuring that the demension equals the dimension. As a consequence of this theorem, we have that any knot (tame or wild) embeds isotopically into the Menger sponge, since its embedding dimension is one.

\subsection{Knots}
 
\noindent A topologically embedded $1$-sphere $K\subset\mathbb{S}^{3}$ is called a {\it topological knot} or {\it 1-knot}. Given two 1-knots $K$, $L:\mathbb{S}^1\rightarrow \mathbb{S}^3$, we say that $K$ is \emph{equivalent} to $L$ if there exists a homeomorphism $\varphi:\mathbb{S}^3\rightarrow\mathbb{S}^3$ such that $\varphi(K)=L$. A 1-knot is {\it tame} if it has a polygonal representative in its ambient isotopy class (see Figure \ref{T1}).\\

\begin{figure}[h] 
\begin{center}
 \includegraphics[height=2.8cm]{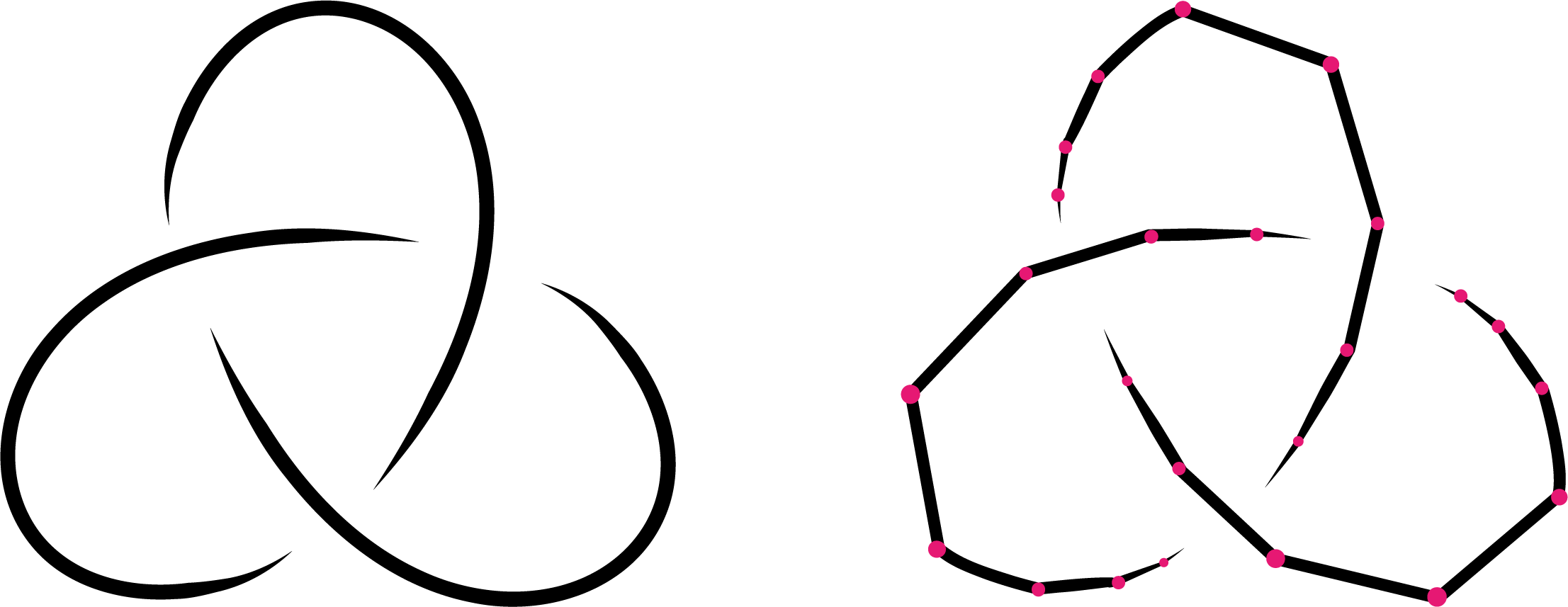}
\end{center}
\caption{\sl A tame knot and its polygonal representative.} 
\label{T1}
\end{figure}

\begin{definition}
We say that a point $x\in K$ is {\it locally flat} or {\it locally tame} if there exist an open neighborhood $U$ of $x$ and a homeomorphism of pairs $(U,U\cap K)\sim (\Int(\mathbb{B}^{3}),\Int(\mathbb{B}^{1}))$, where $\mathbb{B}^{1}\subset\mathbb{R}^1$ is the unit closed $1$-ball (see Figure \ref{flatpt}). Otherwise, $x$ is called a {\it wild} point of $K$. A knot $K$ is {\it locally flat} or {\it locally tame} if all of its points are locally flat. Otherwise, we say $K$ is a {\it wild knot}.
\end{definition}

\begin{figure}[h] 
\begin{center}
\includegraphics[height=2.8cm]{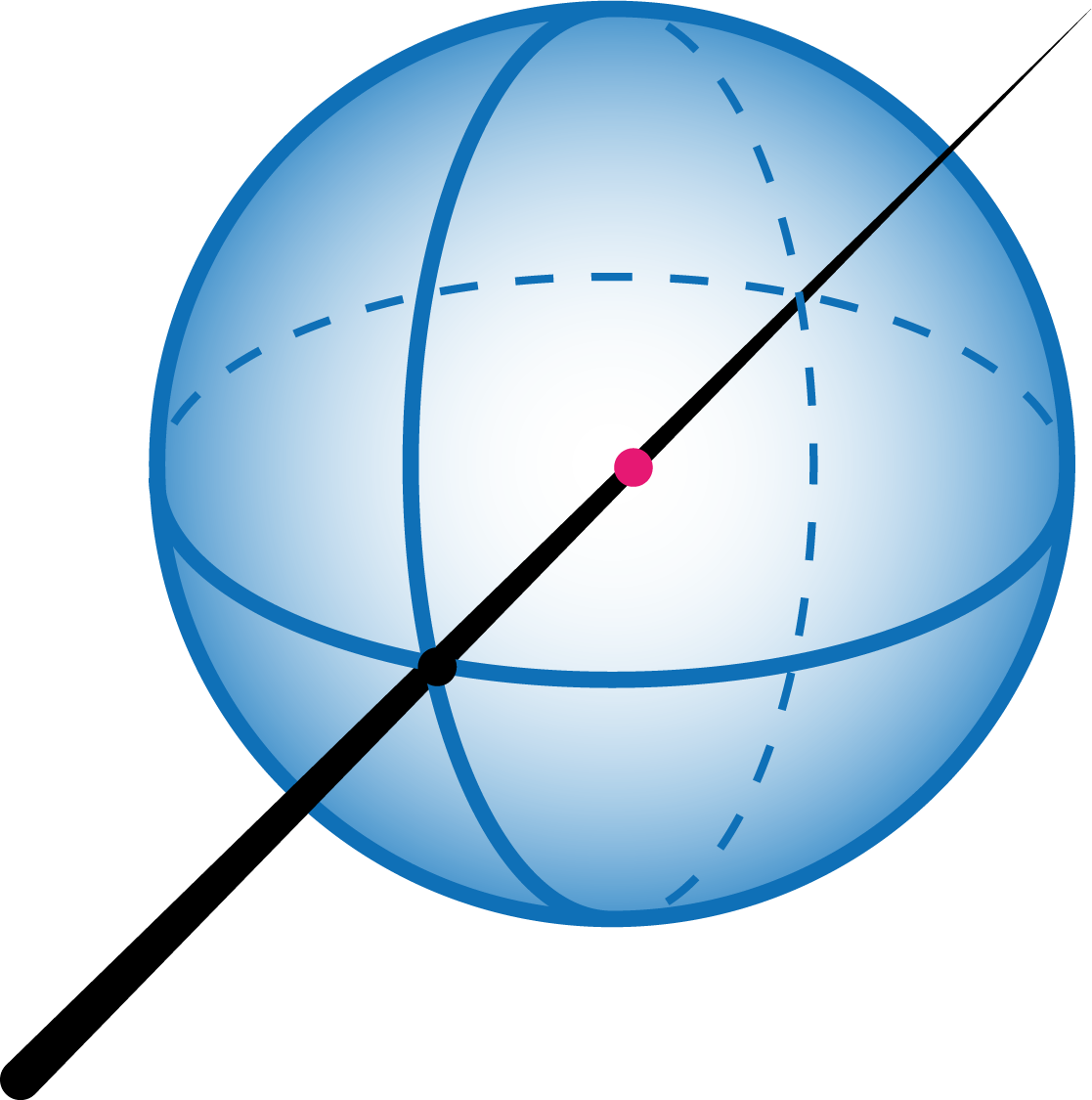}
\end{center}
\caption{\sl A locally tame point.}
\label{flatpt}
\end{figure}

\begin{rem}
If the fundamental group $\pi_1(\mathbb{S}^3-K)$ is infinitely generated, then it can be shown that the knot $K:\mathbb{S}^1\to\mathbb{S}^3$ is wild.
\end{rem}

\noindent We can extend the above notion to any embedding (see \cite{rushing}).

\begin{definition}
Let $X$ be a polyhedron and let $Y$ be a PL-manifold. An embedding $X\rightarrow Y$ is said to be a \emph{tame embedding} if it is equivalent to a PL-embedding; otherwise, it is called a \emph{wild embedding}.
\end{definition}

\noindent {\bf The Arc Presentation}\\

\noindent As in \cite{BENV}, we will also use an arc presentation in grid form for a given tame knot $K$, so we briefly recall it (see \cite{BENV}). An arc presentation in grid form is encoded in an ordered list of $n$ unordered pairs $\{a_1,b_1\},\,\dots,\,\{a_n,b_n\}$, where $a_1,\dots,a_n$ and $b_1,\dots,b_n$ are permutations of the integers $1,\dots,n$ and $a_i\neq b_i$ for $i=1, \dots, n$. Now consider a finite grid of points $(a,b)$, $1\leq a,\,b\leq n$, consisting of $n^2$ points in the lattice $\mathbb{Z}^2$ of $\mathbb{R}^2$, and construct the presentation as follows:

\begin{enumerate}
\item {\it Step 1.} For each $i$ ($1\leq i\leq n$) draw the horizontal line segment joining the point $(a_i,i)$ with the point $(b_i,i)$.
\item {\it Step 2.} Each vertical line at $j$ ($j=1,\,2,\,\ldots,n$) has exactly two points drawn on it. Draw the line segment joining them for each $j$.
\item {\it Step 3.} At each crossing, always draw the corresponding vertical segment above the horizontal one.
\end{enumerate}

\noindent The minimal $n$ for which the above construction can be carried out is called the \emph{arc index of $K$}, and it is denoted by $\alpha (K)$. P. R. Cromwell proved in \cite{cromwell} the following result (compare \cite{BENV}).

\begin{prop}
Every tame knot admits an arc presentation in grid form. Furthermore, there is an algorithm to construct the arc presentation from any other knot diagram of the knot.
\end{prop}

\noindent J. Broden, M. Espinosa, N. Nazareth, and N. Voth proved the following result in \cite{BENV}.
\begin{embedded}
Any tame knot $K$ can be embedded into a finite iteration of the Menger sponge.
\end{embedded}

\noindent{\it Proof.} We provide a sketch of the proof; for more details, see \cite{BENV}. Let $K$ be a tame knot, so $K$ admits an arc presentation
$$
\{a_1,b_1\},\,\dots,\,\{a_n,b_n\}
$$
where $a_1,\,\dots,\,a_n$ and $b_1,\,\dots,\,b_n$ are permutations of $1,\,\dots,\,n$.\\

\noindent Take an iteration of the Cantor set construction (the $k$th stage) that has at least $n$ endpoints. This is possible since the $k$th iteration has $2^{k+1}$ endpoints. Among those endpoints choose $n$, say $p_1<p_2<\cdots <p_n$. Then we consider the arc presentation given by
$$
\{p_{a_1},p_{b_1}\},\,\dots,\,\{p_{a_n},p_{b_n}\}
$$ 
where $p_{a_i}$ is placed on the horizontal axis of the front face, and $p_{b_j}$ on the vertical axis of the front face (with origin at the lower left corner), this diagram has the same knot type.\\

\noindent Let $x_0$ be a point in the Cantor set. By the fact that a point $(x,y)$ belongs to the Sierpi\'nski carpet if and only if $x$ and $y$ do not have a digit 1 in the same position in their non‑ending ternary representation, we have that the vertical segment $(x_0,y,0)$ with $0\leq y\leq 1$ lies entirely on the front face of the Menger sponge $M$. The analogous statement holds for horizontal segments $(x,y_0,1)$ if $y_0$ is an endpoint of the Cantor set. Finally, by Lemma \ref{dust-points}, the endpoints of a vertical segment on the front face can be joined with the corresponding endpoint of the horizontal segment on the back face, since the coordinates of such points are of the form $(x_0,y_0,0)$ and $(x_0,y_0,1)$ with $(x_0,y_0)$ in the Cantor dust (see Figure \ref{knotM}). Therefore, the result follows. $\square$

\begin{figure}[h] 
\begin{center}
\includegraphics[height=6cm]{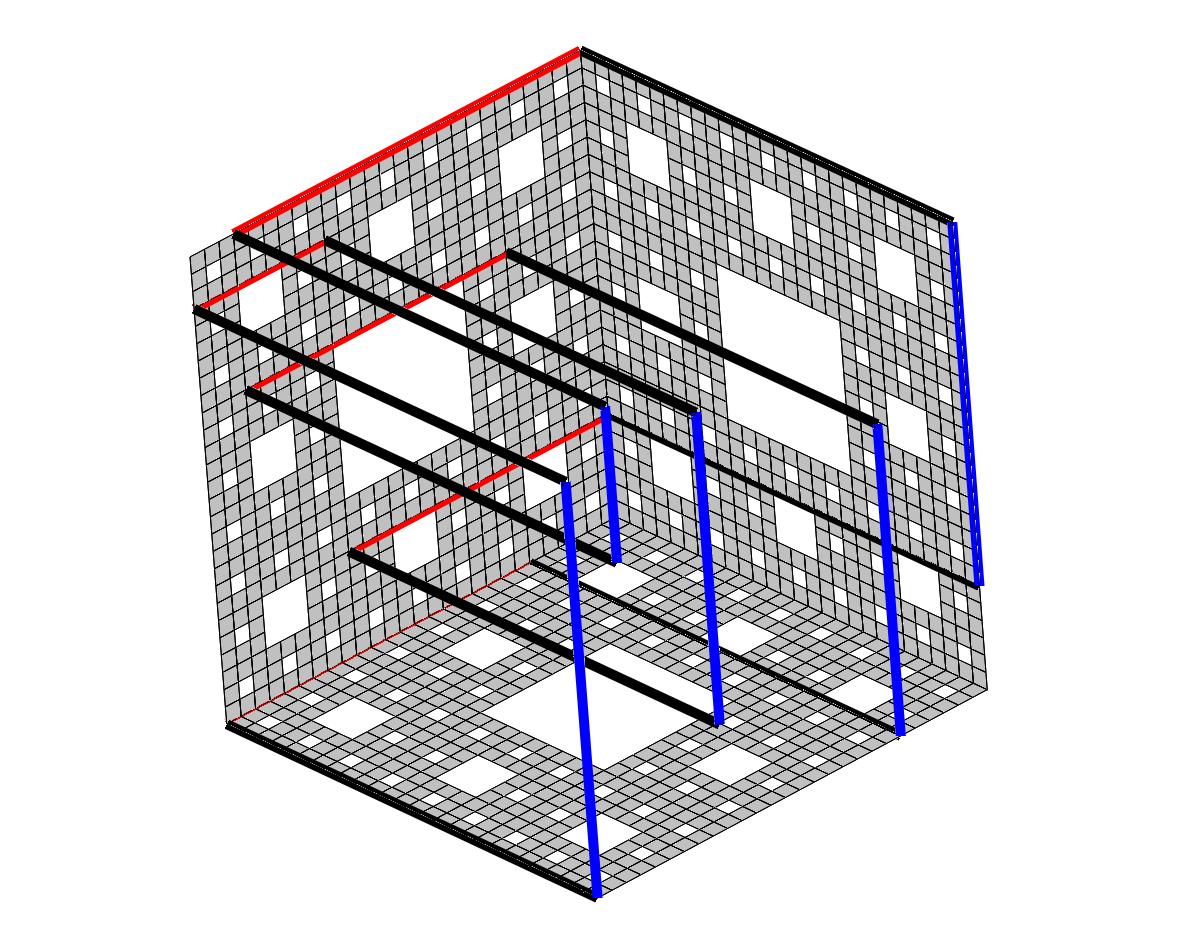}
\end{center}
\caption{\sl The figure-eight knot on the Menger sponge.} 
\label{knotM}
\end{figure}. 

\section{Wild knots embedded in the Menger sponge}

The goal of this section is to prove Theorem 1. We start by constructing squareflake curves.

\subsection{Squareflake curves}

Let $I=[0,1]$ be the unit interval. Consider the unit square $I^2\times\{0\}\subset\mathbb{R}^3$, whose boundary $S$ is a simple closed curve obtained by joining, in order, the vertices $(0,0,0)$, $(0,1,0)$, $(1,1,0)$, and $(1,0,0)$ with four linear segments (see Figure \ref{square}). This curve $S$ is homeomorphic to a circle. We will modify $S$ as follows. Let $e$ be the linear segment joining $(1,1,0)$ and $(1,0,0)$, so $e\subset S$.\\

\begin{figure}[h] 
\begin{center}
\includegraphics[width=5cm, height=4cm ]{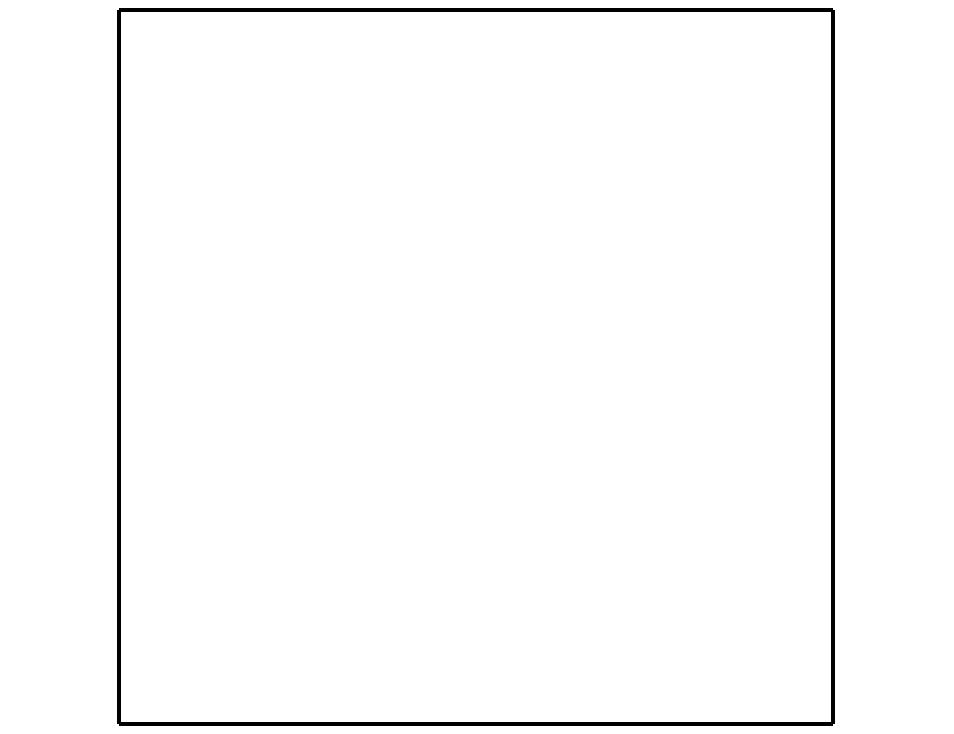}
\includegraphics[width=5cm, height=4cm ]{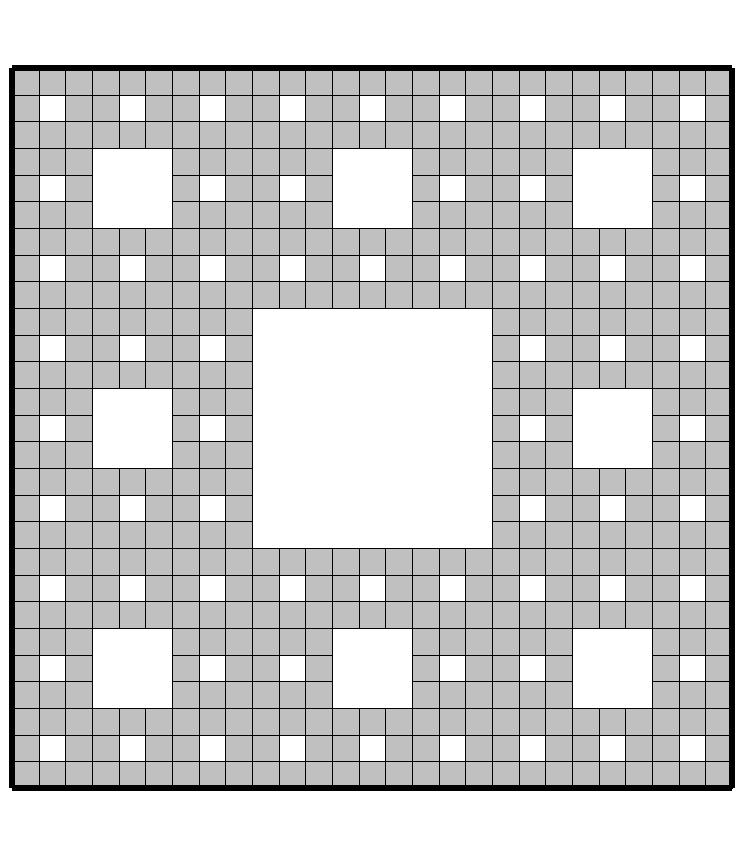}
\caption{\sl The curve $S$.}
\label{square}
\end{center}
\end{figure}

\noindent {\it First stage.} As in the construction of the usual Cantor set, consider the open middle‑third interval $e_1$ in $e$, running from $(1,\frac{1}{3},0)$ to $(1,\frac{2}{3},0)$ and contained in $S$. Let $F_1$ be the square having $e_1$ as one of its sides ($e_1 \subset F_1$), and consider its boundary $\partial F_1$. Replace $e_1$ in $S$ with $\partial F_1\setminus e_1$, keeping its endpoints fixed. Specifically, $e_1$ is replaced by three straight segments: from $(1,\frac{1}{3},0)$ to $(\frac{2}{3},\frac{1}{3},0)$, from $(\frac{2}{3},\frac{1}{3},0)$ to $(\frac{2}{3},\frac{2}{3},0)$, and from $(\frac{2}{3},\frac{2}{3},0)$ to $(1,\frac{2}{3},0)$ (see Figure \ref{F1}). This yields a new curve $S_1$ homeomorphic to $\mathbb{S}^1$. This transformation extends to an ambient isotopy $H_1: \mathbb{R}^3 \times I \rightarrow \mathbb{R}^3$ such that $S_1=H_{1_1}(S)$, where $H_{1_1}(x)=H_1(x,1)$. Moreover, we have $H_{1_1}(S\cap S_1)=S\cap S_1$. By Lemma \ref{dust-points}, $S_1$ is contained in the Menger sponge.\\

\begin{figure}[h] 
\begin{center}
\includegraphics[width=5cm, height=4cm]{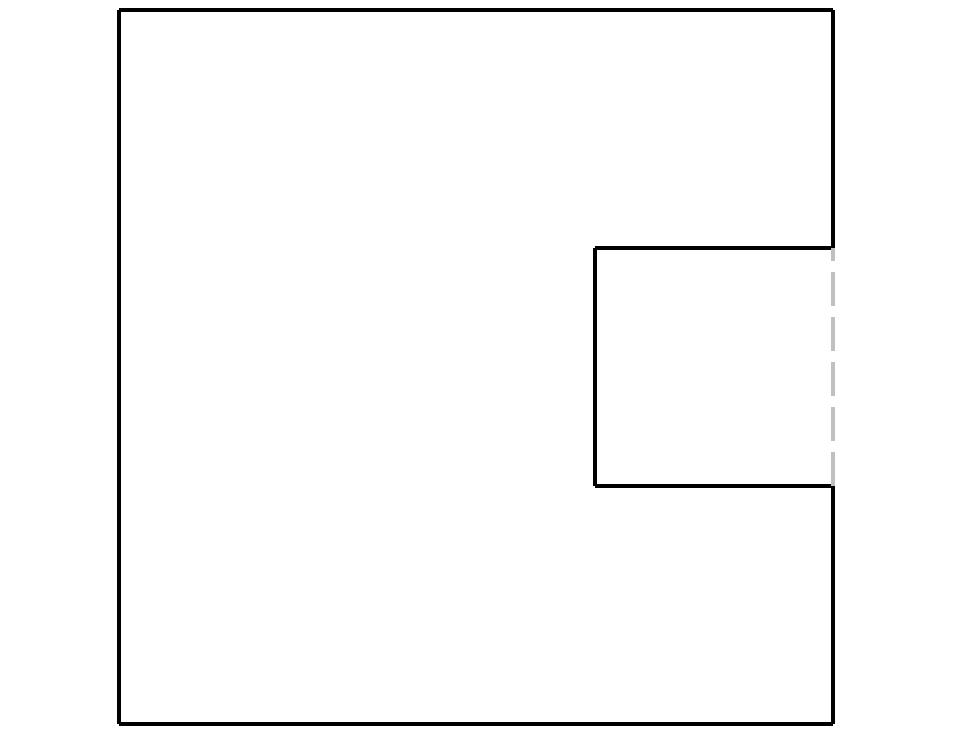}
\includegraphics[width=5cm, height=4cm]{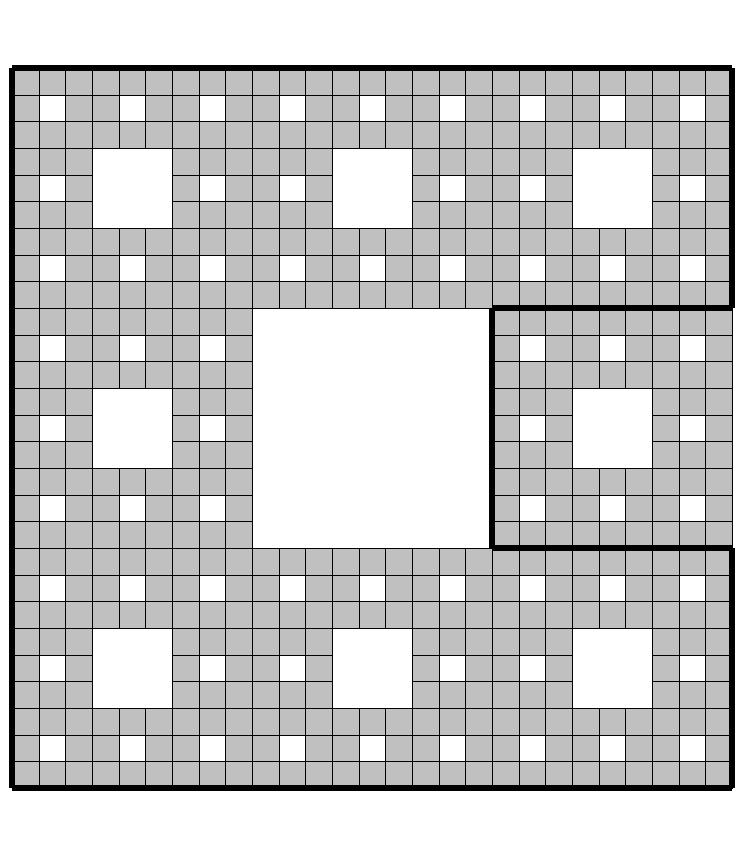}
\caption{\sl The curve $S_1$.}
\label{F1}
\end{center}
\end{figure}

\noindent {\it Second stage.} We now modify $S_1$. As in the first stage, consider the open middle‑third intervals $e_{2_s}$, $s=1,2$, of each of the two remaining segments of the edge $e$. Let $F_{2_s}$ be the square such that $e_{2_s}$ is one of its sides ($e_{2_s} \subset F_{2_s}$). Again, consider the boundary $\partial F_{2_s}$. We replace the segment $e_{2_s} \subset S_1$ with $\partial F_{2_s}  \setminus e_{2_s}$, keeping the endpoints of $e_{2_s}$ fixed. For $e_{2_1}$, these endpoints are $(1,\frac{1}{9},0)$ and $(1,\frac{2}{9},0)$; for $e_{2_2}$, they are $(1,\frac{7}{9},0)$ and $(1,\frac{8}{9},0)$. In other words, we replace $e_{2_1}$ with the union of three segments: from $(1,\frac{1}{9},0)$ to $(\frac{8}{9},\frac{1}{9},0)$, from $(\frac{8}{9},\frac{1}{9},0)$ to $(\frac{8}{9},\frac{2}{9},0)$, and from $(\frac{8}{9},\frac{2}{9},0)$ to $(1, \frac{2}{9},0)$. Similarly, we replace $e_{2_2}$ with three linear segments joining $(1,\frac{7}{9},0)$ to $(\frac{8}{9},\frac{7}{9},0)$, then to $(\frac{8}{9},\frac{8}{9},0)$, and finally to $(1, \frac{8}{9},0)$. Thus, we obtain a new curve $S_2$, again homeomorphic to $\mathbb{S}^1$. As before, there exists an ambient isotopy $H_2: \mathbb{R}^3\times I \rightarrow \mathbb{R}^3$ such that $S_2 = H_{2_1}(S_1)$, where $H_{2_1}(x) = H_2(x, 1)$. Again we have $H_{2_1}(S_1\cap S_2)=S_1\cap S_2$. This curve is also contained in the Menger sponge.\\

\begin{figure}[h] 
\begin{center}
\includegraphics[width=5cm, height=4cm]{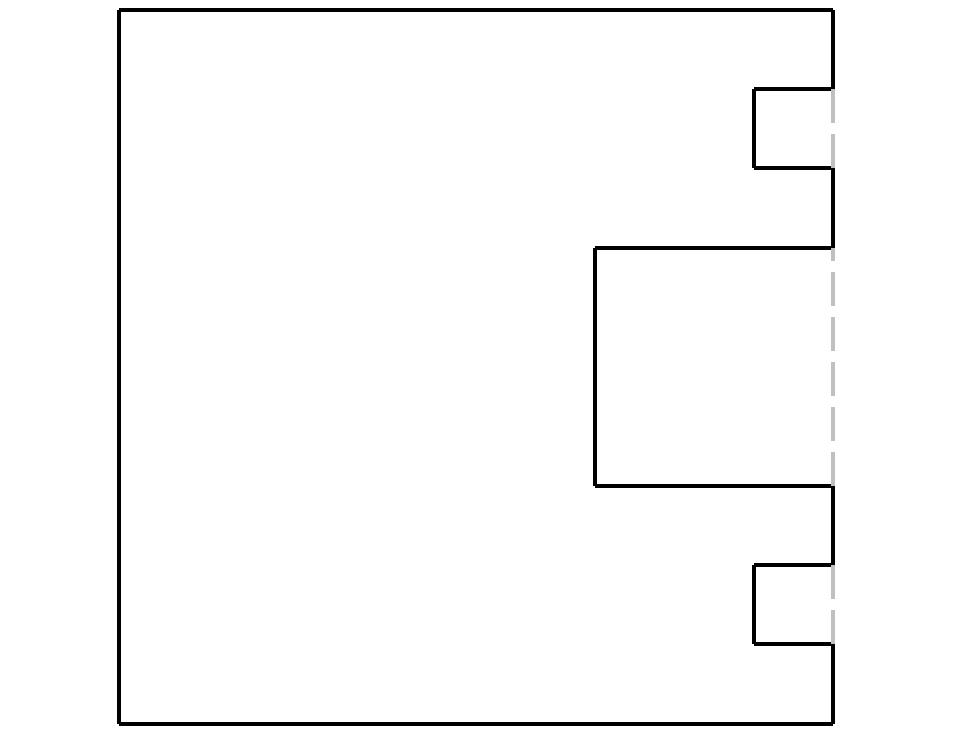}
\includegraphics[width=5cm, height=4cm]{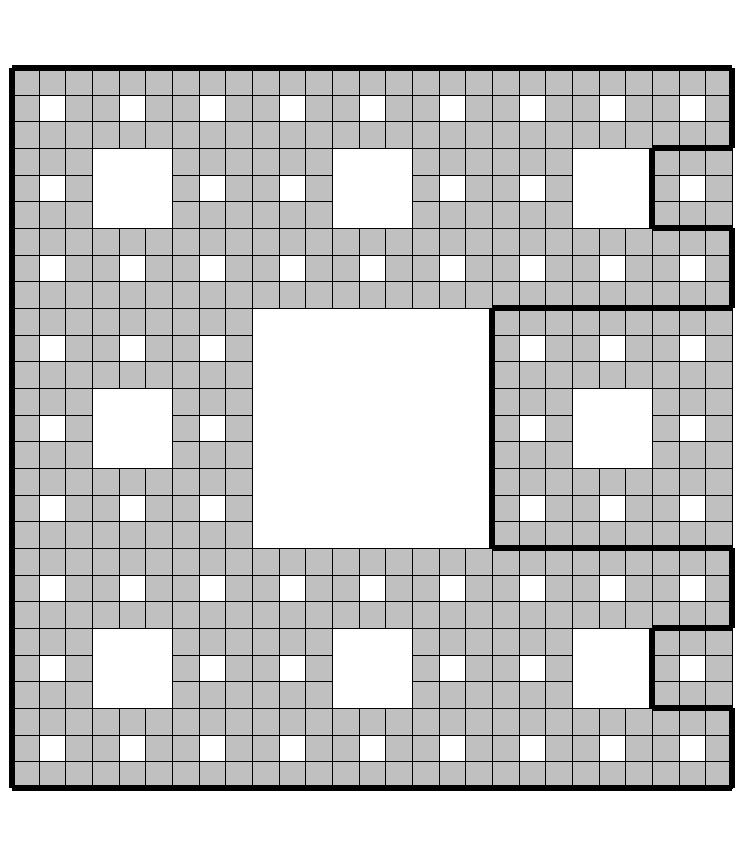}
\caption{\sl The curve $S_2$.}
\label{F2}
\end{center}
\end{figure}

\noindent {\it $m^{th}$ stage}. Let
\begin{equation*}
\begin{split}
D_{m-1}&=\Big\{\,k=\textstyle\sum_{j=1}^{m-1} d_j\,3^{\,j-1} \ :\ d_j\in\{0,2\}\ \text{for } j=1,\ldots,m-1\,\Big\}\\
&\subset\{0,1,\ldots,3^{m-1}-1\},
\end{split}
\end{equation*}
so that $|D_{m-1}|=2^{m-1}$; this is exactly the (finite) set of indices $k$ for which the segment $[\frac{3k+1}{3^m},\frac{3k+2}{3^m}]$ is a middle third that has \emph{not yet been touched} by stages $1,\ldots,m-1$ (for $m=1$ this is the single index $k=0$, recovering the first stage; for $m=2$ it is $\{0,2\}$, recovering the second stage). We repeat the same process on each open middle-third interval $e_{m_s}$, $s=1,\ldots,2^{m-1}$, indexed by $k\in D_{m-1}$, of each of the remaining segments of $e$ obtained at the $(m-1)$-th stage. More precisely, for each $k\in D_{m-1}$ we replace the segment joining $(1,\frac{3k+1}{3^m},0)$ to $(1,\frac{3k+2}{3^m},0)$ by the union of three segments joining
$(1,\frac{3k+1}{3^m},0)$ to $(\frac{3^m-1}{3^m},\frac{3k+1}{3^m},0)$, $(\frac{3^m-1}{3^m},\frac{3k+1}{3^m},0)$ to $(\frac{3^m-1}{3^m},\frac{3k+2}{3^m},0)$, and $(\frac{3^m-1}{3^m},\frac{3k+2}{3^m},0)$ to
$(1, \frac{3k+2}{3^m},0)$ (see Figure \ref{F3}). The new curve $S_m$ obtained from $S_{m-1}$ is homeomorphic to $\mathbb{S}^1$. Again, there exists an ambient isotopy $H_m: \mathbb{R}^3\times I \rightarrow \mathbb{R}^3$ such that $S_m = H_{m_1}(S_{m-1})$, where $H_{m_1}(x) = H_m(x, 1)$. Moreover, $H_{m_1}(S_{m-1}\cap S_m)=S_{m-1}\cap S_m$. This curve is also contained in the Menger sponge. Note that once a middle third has been replaced by a notch at stage $m$, it is never modified again at a later stage: stage $m+1$ only acts on the $2^m$ middle thirds of the segments of $e$ that are \emph{still untouched} after stage $m$. This bookkeeping detail is the key to Lemma \ref{curve} below.\\

\begin{figure}[h] 
\begin{center}
\includegraphics[width=5cm, height=4cm]{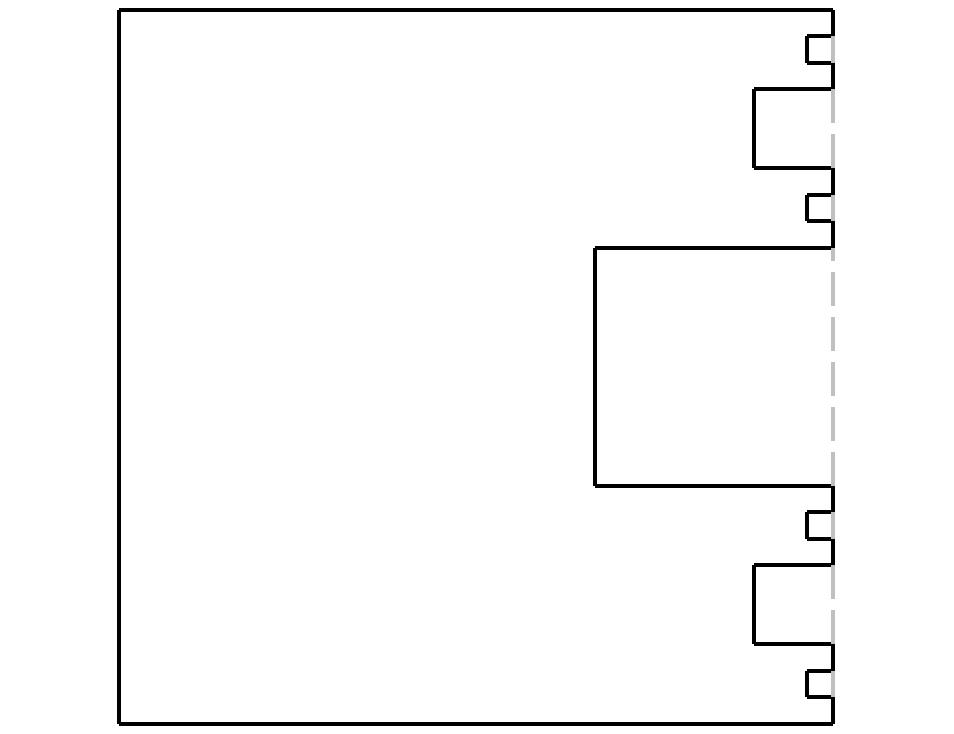}
\includegraphics[width=5cm, height=4cm]{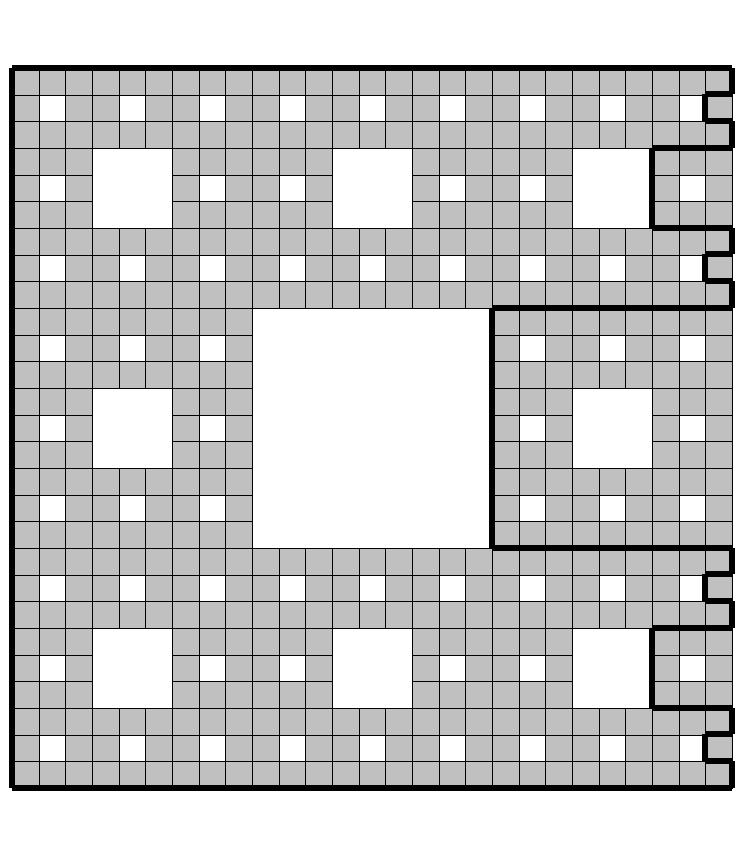}
\caption{\sl The curve $S_3$.}
\label{F3}
\end{center}
\end{figure}

\noindent We now define the limiting object $\mathcal S=\varinjlim_k S_k$ directly as a subset of $\mathbb{R}^3$, rather than only as an abstract direct limit; this is what is needed to make Lemma \ref{curve} below precise.

\medskip
\noindent{\bf A locally finite replacement lemma.} We isolate the elementary point-set principle underlying the construction, since we will use it three times in this paper (for $\mathcal S$, for the wild knots $\mathcal K$ of Section 3.2, and for the isotopy of Theorem 2 in Section 3.3).

\begin{lem}\label{embed}
Let $J=[0,1]$ and let $C\subset J$ be the standard middle-thirds Cantor set, obtained by removing the countable family of pairwise disjoint open intervals $\{(a_i,b_i)\}_{i\in\mathbb{N}}$ ($J\setminus C=\bigsqcup_i (a_i,b_i)$). Let $\phi_0:J\rightarrow \mathbb{R}^3$ be a topological embedding. Suppose that for each $i$ we are given a topological embedding $\gamma_i:[a_i,b_i]\rightarrow\mathbb{R}^3$ such that:
\begin{enumerate}
\item[(a)] $\gamma_i(a_i)=\phi_0(a_i)$ and $\gamma_i(b_i)=\phi_0(b_i)$;
\item[(b)] $\mathrm{diam}\,\gamma_i([a_i,b_i])\leq \kappa\,(b_i-a_i)$ for some constant $\kappa\geq1$ independent of $i$;
\item[(c)] the sets $\gamma_i((a_i,b_i))$, $i\in\mathbb{N}$, are pairwise disjoint and disjoint from $\phi_0(J)$.
\end{enumerate}
Define $\phi:J\rightarrow\mathbb{R}^3$ by $\phi=\phi_0$ on $C$ and $\phi=\gamma_i$ on $[a_i,b_i]$ for each $i$ (this is consistent by (a)). Then $\phi$ is a topological embedding, with image
$$
\phi(J)=\phi_0(C)\,\cup\,\bigcup_{i\in\mathbb{N}}\gamma_i([a_i,b_i]).
$$
\end{lem}

\noindent{\it Proof.} \emph{Well-defined and continuous.} $\phi$ is well defined by (a). To see $\phi$ is continuous, it suffices to check continuity at points $x_0\in C$ (continuity on the open set $J\setminus C=\bigsqcup(a_i,b_i)$ is clear, since $\phi$ restricted to each closed interval $[a_i,b_i]$ is the continuous map $\gamma_i$). Let $x_0\in C$ and $\varepsilon>0$. Since $C$ is nowhere dense and $\{(a_i,b_i)\}$ is a null sequence of intervals (standard fact for the middle-thirds Cantor set: for every $\delta>0$ only finitely many $i$ have $b_i-a_i\geq\delta$, because $\sum_i(b_i-a_i)=1<\infty$), we may choose $\delta>0$ so small that every interval $(a_i,b_i)$ meeting $(x_0-\delta,x_0+\delta)$ has length $b_i-a_i<\varepsilon/(2\kappa+2)$, and so small that $\phi_0$ maps $(x_0-\delta,x_0+\delta)$ into the $\varepsilon/2$-ball about $\phi_0(x_0)$ (possible since $\phi_0$ is continuous). Now let $x\in J$ with $|x-x_0|<\delta$. If $x\in C$, then $|\phi(x)-\phi(x_0)|=|\phi_0(x)-\phi_0(x_0)|<\varepsilon/2<\varepsilon$. If $x\in(a_i,b_i)$ for some $i$, then $(a_i,b_i)\subset(x_0-2\delta,x_0+2\delta)$ once $\delta$ is small enough (since $x$ is within $\delta$ of $x_0$ and $b_i-a_i<\delta$ for $\delta$ small), so $a_i\in(x_0-2\delta,x_0+2\delta)$, and by (a) and (b),
\begin{equation*}
\begin{split}
|\phi(x)-\phi(x_0)|&\leq \mathrm{diam}\,\gamma_i([a_i,b_i])+|\phi_0(a_i)-\phi_0(x_0)|\\
&\leq \kappa(b_i-a_i)+|\phi_0(a_i)-\phi_0(x_0)|,
\end{split}
\end{equation*}
which is $<\varepsilon$ once $\delta$ is chosen small enough that both terms are controlled as above. Hence $\phi$ is continuous at $x_0$.

\smallskip
\noindent \emph{Injective.} If $x,x'\in C$, $\phi(x)=\phi(x')$ forces $x=x'$ since $\phi_0$ is injective. If $x\in(a_i,b_i)$, $x'\in(a_j,b_j)$ with $i\neq j$, then $\phi(x)\in\gamma_i((a_i,b_i))$ and $\phi(x')\in\gamma_j((a_j,b_j))$ are distinct by (c). If $x\in(a_i,b_i)$ and $x'\in C$, then $\phi(x)\in\gamma_i((a_i,b_i))$ while $\phi(x')\in\phi_0(J)$, again distinct by (c) (using $\phi_0(x')\in\phi_0(J)$). If $x,x'\in[a_i,b_i]$ for the same $i$, injectivity follows from injectivity of $\gamma_i$.

\smallskip
\noindent \emph{Homeomorphism onto its image.} $\phi$ is a continuous injection from the compact space $J$ into the Hausdorff space $\mathbb{R}^3$, hence a closed map onto its image, and thus a homeomorphism onto its image. $\square$

\medskip
\noindent We now apply Lemma \ref{embed} to the squareflake construction. Identify $e$ with $J=[0,1]$ via $y\mapsto(1,y,0)$, and let $\phi_0:J\to e\subset\mathbb R^3$, $\phi_0(y)=(1,y,0)$, be this identification; let $C\subset[0,1]$ be the standard Cantor set, so that $J\setminus C=\bigsqcup_{m\geq1}\bigsqcup_{k\in D_{m-1}}\big(\tfrac{3k+1}{3^m},\tfrac{3k+2}{3^m}\big)$, where $D_{m-1}$ is the index set defined above. For the interval $(a,b)=\big(\tfrac{3k+1}{3^m},\tfrac{3k+2}{3^m}\big)$ ($k\in D_{m-1}$), let $\gamma_{m,k}:[a,b]\to\mathbb{R}^3$ be the affine, constant-speed parametrization of the three-segment notch built at the $m$-th stage of the construction above (i.e., the first, second and third thirds of $[a,b]$ are mapped affinely onto the three straight segments described in the $m$-th stage). Since the notch is contained in the box $[\frac{3^m-1}{3^m},1]\times[\frac{3k+1}{3^m},\frac{3k+2}{3^m}]\times\{0\}$, whose diameter is at most $\sqrt2\cdot\frac{1}{3^m}=\sqrt2\,(b-a)$, hypothesis (b) of Lemma \ref{embed} holds with $\kappa=\sqrt2$; hypotheses (a) and (c) hold by construction, since distinct stage-$m$ notches have disjoint $y$-ranges $[\frac{3k+1}{3^m},\frac{3k+2}{3^m}]$ (these are exactly the removed middle thirds, which are pairwise disjoint by the standard Cantor construction) and every notch is confined to $x\geq\frac{3^m-1}{3^m}>\frac23$, hence disjoint from the interior of $e$ near $\phi_0(J)$ except at its two prescribed endpoints. Lemma \ref{embed} therefore produces an embedding $\phi_e:J\to\mathbb{R}^3$ with image $\phi_0(C)\cup\bigcup_{m,k}\gamma_{m,k}([a,b])$, i.e., the Cantor set of un-notched points of $e$ together with all the notches built at every stage.

\noindent Finally, extend $\phi_e$ to an embedding $\phi:\mathbb{S}^1\to\mathbb{R}^3$ of the full boundary circle by keeping the other three sides of the square fixed (they are never modified), matching endpoints at the vertices $(0,0,0),(0,1,0),(1,1,0)$. Define
$$
\mathcal{S}:=\phi(\mathbb{S}^1)= (S\setminus e)\,\cup\, \phi_0(C)\,\cup\,\bigcup_{m\geq1}\bigcup_{k\in D_{m-1}}\gamma_{m,k}\Big(\big[\tfrac{3k+1}{3^m},\tfrac{3k+2}{3^m}\big]\Big).
$$
This is the honest, point-set definition of the squareflake curve as a subset of $\mathbb{R}^3$; the informal description ``$\mathcal S=\varinjlim_k S_k$'' used below refers to this object, justified by the following lemma.

\begin{lem}\label{curve}
The set ${\mathcal{S}}$ defined above is the image of a topological embedding $\phi:\mathbb{S}^1\to\mathbb{R}^3$; in particular ${\mathcal{S}}$ is homeomorphic to $\mathbb{S}^1$. Moreover, 
$S_m\to\mathcal S$ in the Hausdorff metric as $m\to\infty$, and ${\mathcal{S}}$ is contained in the Menger sponge.
\end{lem}
\noindent{\it Proof.} The first statement is Lemma \ref{embed}, applied as above. For the Hausdorff convergence: $S_m$ and $\mathcal S$ agree outside the union of the (at most countably many, but only finitely many of diameter $\geq\varepsilon$) notch-boxes of stage $>m$, all of which have diameter at most $\sqrt2\cdot 3^{-(m+1)}$; hence the Hausdorff distance between $S_m$ and $\mathcal S$ is at most $\sqrt2\cdot3^{-(m+1)}\to0$. Finally, each point of $\mathcal S$ lies either in $S\setminus e$ (contained in $M$ trivially, as it lies on the coordinate axes with Cantor-compatible ternary expansions, as in the proof of Theorem A), in $\phi_0(C)$ (contained in $M$ by Lemma \ref{dust-points}, since $x=1=0.\overline{2}_3$ and $y\in C$ both have ternary expansions avoiding the digit $1$), or in some notch $\gamma_{m,k}([a,b])$, which was already shown, stage by stage, to lie in $M$ in the construction above. Hence $\mathcal S\subset M$. $\square$

\begin{definition}
Any curve ${\mathcal{S}}$ constructed as above, using either the usual Cantor set or a generalization of it, is called a \emph{squareflake curve}.
\end{definition}

\subsection{Main theorem}

Let us prove Theorem 1. We begin by constructing wild knots in the Menger sponge.\\

\noindent Let $I=[0,1]$ be the unit interval, and consider the usual Menger sponge $M$ contained in $I^3$. Notice that there is a usual Cantor set contained in the edge $e=\{(1,t,0)\,\,:\,\,t\in [0,1]\}\subset M$. We will construct our wild knots recursively, compatible with the construction of the Menger sponge. Consider the face $F=I^2\times\{0\}$ such that the squareflake curve $\mathcal{S}$ lies entirely on $F$.\\

\noindent {\bf Construction:} Let $K$ be a tame knot. By Theorem A, there exists a stage $n(K)$ of the Menger sponge construction such that $K$ is contained in $M_{n(K)}$.\\

\noindent {\it First stage.} Consider the squareflake curve $\mathcal{S}$ and the second stage of the Menger sponge, $M_2$. There are three cubes $Q_{1_1}$, $Q_{1_2}$ and $Q_{1_3}$ in $M_2$ containing, as sides, the middle‑third segments $T_{1_1}=[\frac{7}{9},\frac{8}{9}]\times\{\frac{1}{3}\}\times\{0\}$, $T_{1_2}=\{\frac{2}{3}\}\times [\frac{4}{9},\frac{5}{9}]\times\{0\}$ and $T_{1_3}=[\frac{7}{9},\frac{8}{9}]\times\{\frac{2}{3}\}\times\{0\}$, respectively (see Figure \ref{W1}). Each $T_{1_i}$ lies entirely on $\mathcal{S}$. Each cube $Q_{1_i}$ is homothetically equivalent to the Menger sponge $M$; hence, by the proof of Theorem A, there exists an isotopic copy of our knot $K$, say $K_{1_i}$, contained in $M\cap Q_{1_i}$, $i=1,2,3$. Furthermore, we can construct $K_{1_i}$ so that $K_{1_i}\cap K_{1_j}=\emptyset$ for $i\neq j$ and an unknotted segment $e_{1_i}$ of $K_{1_i}$ is contained in the corresponding middle‑third segment $T_{1_i}$. This allows us to define $K_{1}$ as the connected sum of $\mathcal{S}$ with $K_{1_1}$, $K_{1_2}$ and $K_{1_3}$ along the corresponding segments $e_{1_i}$; that is,
$$
K_1\cong {\mathcal{S}}\#K_{1_1}\# K_{1_2}\# K_{1_3}.   
$$
Observe that $K_1$ is a knot lying entirely in the Menger sponge, since ${\mathcal{S}}$ and each $K_{1_i}$ do. \\

\begin{figure}[h] 
\begin{center}
\includegraphics[width=5cm, height=4cm]{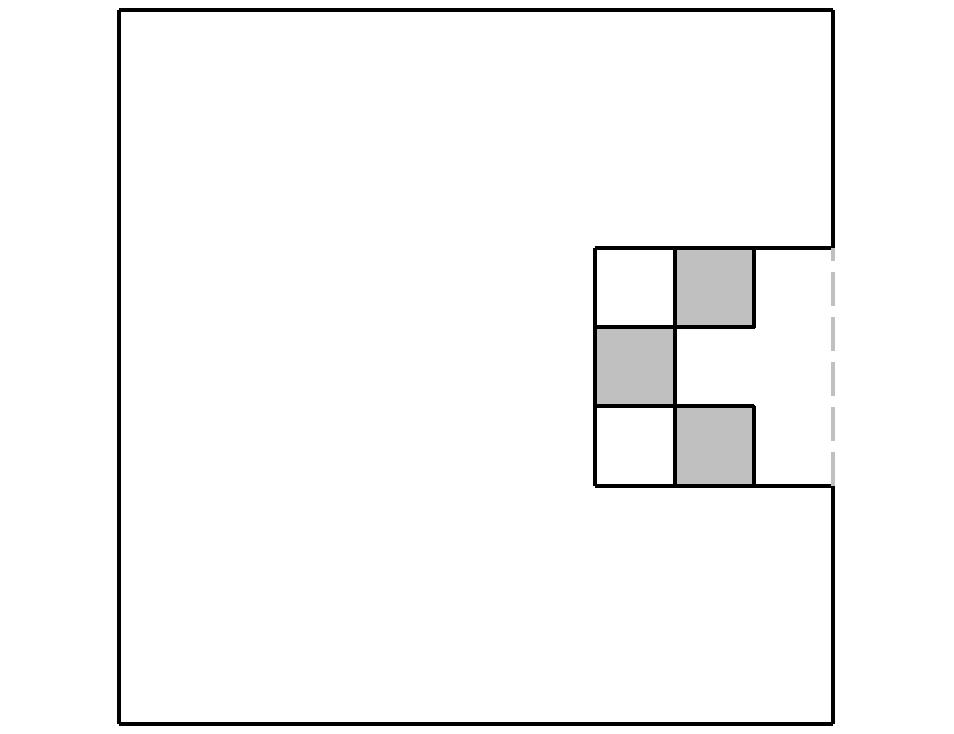}
\includegraphics[width=5cm, height=4cm]{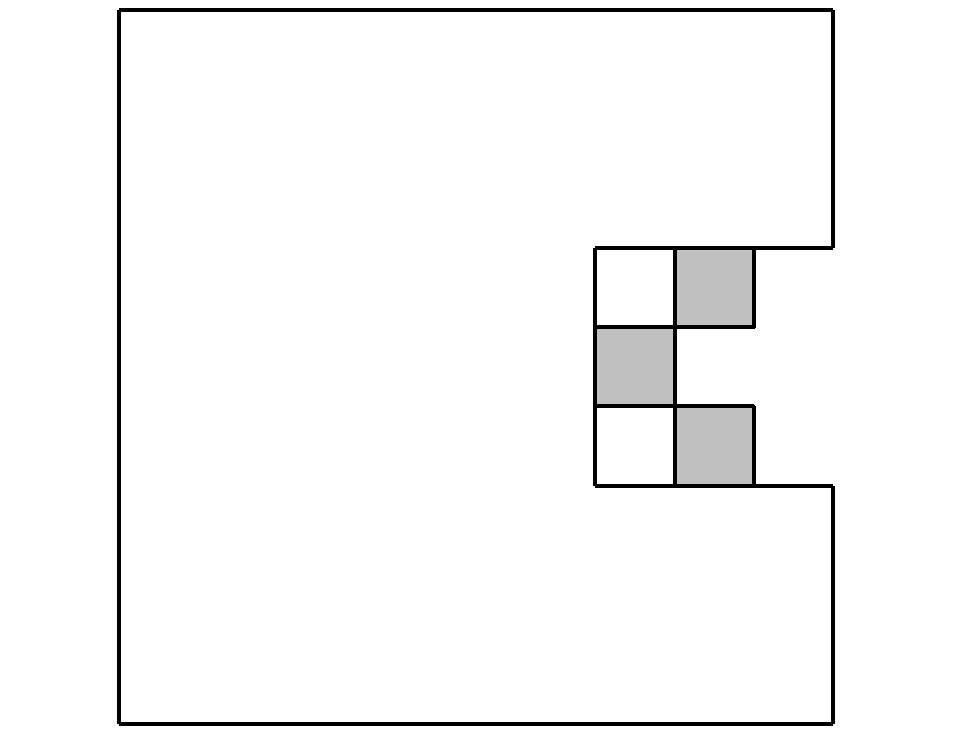}
\caption{\sl The curve $S_1$ and the projection of the cubes $Q_{1_1}$, $Q_{1_2}$ and $Q_{1_3}$ contained in the second stage of the Menger sponge, $M_2$.}
\label{W1}
\end{center}
\end{figure}

\noindent {\it Second stage.} Consider our knot $K_1$ and the third stage of the Menger sponge, $M_3$. There are six cubes $Q_{2_i}$ ($i=1,2,\ldots,6$) in $M_3$ containing, as sides, the middle‑third segments $T_{2_1}=[\frac{25}{3^3},\frac{26}{3^3}]\times\{\frac{1}{9}\}\times\{0\}$, $T_{2_2}=\{\frac{8}{3^2}\}\times [\frac{4}{3^3},\frac{5}{3^3}]\times\{0\}$, $T_{2_3}=[\frac{25}{3^3},\frac{26}{3^3}]\times\{\frac{2}{9}\}\times\{0\}$, $T_{2_4}=[\frac{25}{3^3},\frac{26}{3^3}]\times\{\frac{7}{9}\}\times\{0\}$, $T_{2_5}=\{\frac{8}{3^2}\}\times [\frac{22}{3^3},\frac{23}{3^3}]\times\{0\}$, and $T_{2_6}=[\frac{25}{3^3},\frac{26}{3^3}]\times\{\frac{8}{9}\}\times\{0\}$, respectively (see Figure \ref{W2}). Again, each $T_{2_i}$ lies entirely on $\mathcal{S}$. Each cube $Q_{2_i}$ is homothetically equivalent to $M$; by the proof of Theorem A, there exists an isotopic copy of our knot $K$, say $K_{2_i}$, contained in $M\cap Q_{2_i}$, $i=1,2,\ldots,6$. We can arrange that $K_{2_i}\cap K_{2_j}=\emptyset$ for $i\neq j$ and that an unknotted segment $e_{2_i}$ of $K_{2_i}$ lies in the corresponding middle‑third segment $T_{2_i}$. We define $K_{2}$ as the connected sum of $K_1$ with $K_{2_1}$, $K_{2_2},\,\ldots,\,K_{2_6}$ along the corresponding segments $e_{2_i}$; that is,
$$
K_2\cong K_1\#K_{2_1}\# K_{2_2}\# K_{2_3}\#K_{2_4}\# K_{2_5}\# K_{2_6}.   
$$
Again, $K_2$ lies entirely in the Menger sponge, since $K_1$ and each $K_{2_i}$ do.\\

\begin{figure}[h] 
\begin{center}
\includegraphics[width=5cm, height=4cm]{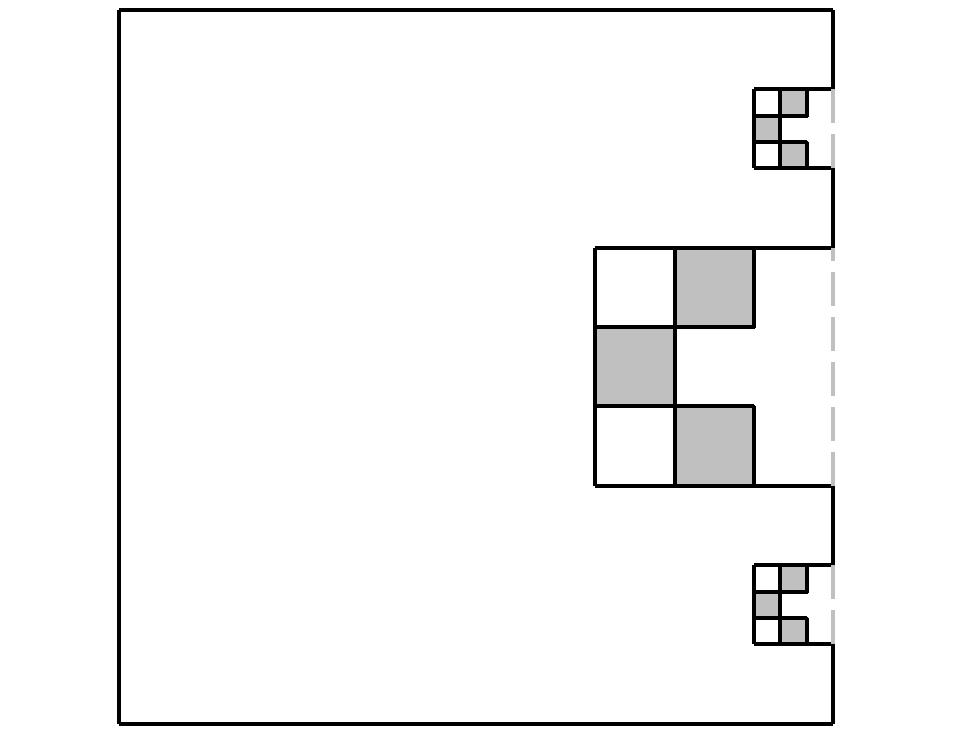}
\includegraphics[width=5cm, height=4cm]{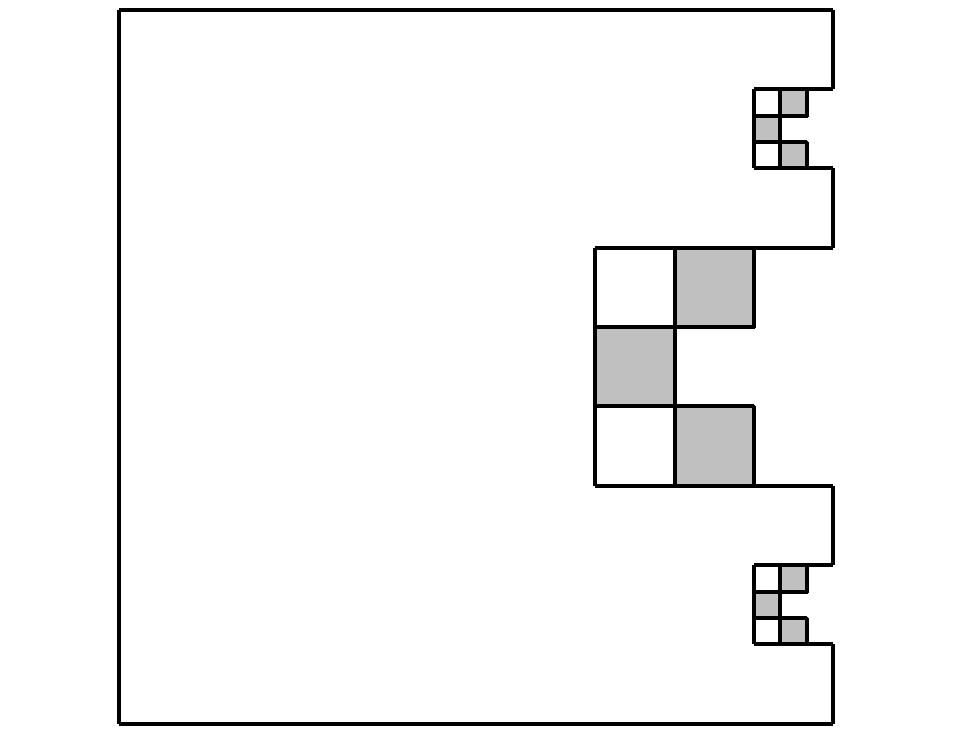}
\caption{\sl The curve $S_2$ and the projection of six cubes contained in the third stage of the Menger sponge, $M_3$.}
\label{W2}
\end{center}
\end{figure}

\noindent {\it $m^{th}$ stage}. Recall the index set $D_{m-1}$ of Section 3.1, $|D_{m-1}|=2^{m-1}$. We repeat the same process on each middle-third interval $T_{m_i}$, $i=1,\ldots,3\cdot2^{m-1}$, corresponding to the three segments $[1-\frac{1}{3^m},1]\times\{\frac{3k+1}{3^m}\}\times\{0\}$, $\{1-\frac{1}{3^m}\}\times[\frac{3k+1}{3^m},\frac{3k+2}{3^m}]\times\{0\}$, and $[1-\frac{1}{3^m},1]\times\{\frac{3k+2}{3^m}\}\times\{0\}$, one triple for each $k\in D_{m-1}$ (\emph{i.e.}, exactly for the $2^{m-1}$ notches newly built at the $m$-th stage of the squareflake curve, each notch contributing its three sides). Each $T_{m_i}$ lies entirely on $\mathcal{S}$. Consider the corresponding cubes $Q_{m_i}$ contained in the $(m+1)$-th stage of the Menger sponge, $M_{m+1}$. Each $Q_{m_i}$ is again homothetically equivalent to $M$. By the proof of Theorem A, there exists an isotopic copy of our knot $K$, say $K_{m_i}$, contained in $M\cap Q_{m_i}$, and we can arrange that $K_{m_i}\cap K_{m_j}=\emptyset$ for $i\neq j$ and that an unknotted segment $e_{m_i}$ of $K_{m_i}$ lies in the corresponding middle-third segment $T_{m_i}$. We define $K_{m}$ as the connected sum of $K_{m-1}$ with $K_{m_1}$, $K_{m_2},\,\ldots,\,K_{m_l}$, where
$l=3\cdot2^{m-1}$ (so that the counts match the first two stages exhibited above: $l=3$ for $m=1$, $l=6$ for $m=2$), along the corresponding segments $e_{m_i}$; that is,
$$
K_m\cong K_{m-1}\#K_{m_1}\# K_{m_2}\#\,\cdots\,\# K_{m_l}.
$$
Again, $K_m$ lies entirely in the Menger sponge, since $K_{m-1}$ and each $K_{m_i}$ do. Exactly as in Section 3.1, once a segment $T_{m_i}$ has been used for an insertion at stage $m$, it is never touched again at a later stage: stage $m+1$ only performs insertions on the $3\cdot2^m$ fresh sides of the $2^m$ notches newly built at the $(m+1)$-th stage of $\mathcal S$. \\

\begin{rems}\label{differentknots}
1. In the construction above, each knot $K_{m_i}$ can be a different tame knot for each $m_i$.\\
2. We can also add more knots at each step. Indeed, we can subdivide each middle-third segment into $2n+1$ subsegments, choose $n$ non-consecutive subsegments $T_{1_j}$, and proceed as before, now taking subcubes $Q_{i_k}$ ($k=1,2,\ldots,n$) of the $(2n+1)$-st stage of the Menger sponge. This variant is used, with $n$ depending on the stage, in the proof of Theorem 2 in Section 3.3.
\end{rems}

\noindent {\bf The embedded limit $\mathcal K$.} As in Section 3.1, we now define $\mathcal K$ as an honest subset of $\mathbb R^3$, using Lemma \ref{embed}. Identify $\mathbb S^1$ with $\mathcal S$ via the embedding $\phi$ of Lemma \ref{curve}. The segments $T_{m_i}$, $m\geq1$, $i=1,\ldots,3\cdot2^{m-1}$, are pairwise disjoint closed sub-arcs of $\mathcal S$ (this is immediate from the disjointness of the notches established in Section 3.1: distinct notches have disjoint closures except at the two points shared with $\phi_0(C)$, and the three sides of a single notch meet only at their common corners), and they accumulate only on the Cantor set $\mathcal C\subset e$: for every $\varepsilon>0$, only finitely many $T_{m_i}$ have diameter $\geq\varepsilon$, since $\mathrm{diam}(T_{m_i})\leq 3^{-m}$. Exactly as in the proof of Lemma \ref{curve}, replace each $T_{m_i}$ by a constant-speed parametrization $\gamma_{m_i}$ of $e_{m_i}\cup(K_{m_i}\setminus e_{m_i})$, \emph{i.e.}, of the arc obtained from $K_{m_i}$ by cutting along the unknotted segment $e_{m_i}\subset T_{m_i}$ and inserting the resulting knotted arc into $T_{m_i}$ (this is precisely the connected-sum replacement described above). Since $K_{m_i}\subset Q_{m_i}$ and $Q_{m_i}$ is a cube of side length $3^{-(m+1)}$, comparable to $\mathrm{diam}(T_{m_i})=3^{-(m+1)}$, hypothesis (b) of Lemma \ref{embed} (or rather its extension below) holds with a fixed constant $\kappa$. Lemma \ref{embed} -- whose proof used only continuity/injectivity of the base map, pairwise disjointness of the replaced arcs, and the diameter bound, none of which depend on the base domain being literally $[0,1]$ with $C$ the triadic Cantor set -- applies verbatim with $J=\mathbb S^1$, $\phi_0=\phi$, $C=\mathcal C$, and $\{(a_i,b_i)\}$ replaced by $\{T_{m_i}\}_{m,i}$, and produces an embedding $\Phi:\mathbb S^1\to\mathbb R^3$. We define
$$
\mathcal K:=\Phi(\mathbb S^1)=\mathcal S\setminus\Big(\bigcup_{m,i}T_{m_i}\Big)\ \cup\ \mathcal C\ \cup\ \bigcup_{m,i}\gamma_{m_i}(T_{m_i}).
$$
This is the honest point-set definition of the wild knot; as before, ``$\mathcal K=\varinjlim_m K_m$'' is informal shorthand for this object, made precise by the next lemma.

\begin{lem}\label{Kembed}
The set ${\mathcal{K}}$ defined above is the image of a topological embedding $\Phi:\mathbb S^1\to\mathbb R^3$; in particular, it is homeomorphic to $\mathbb{S}^1$. Moreover, $K_m\to\mathcal K$ in the Hausdorff metric as $m\to\infty$, at rate $O(3^{-m})$, and $\mathcal K$ is contained in the Menger sponge.
\end{lem}
\noindent{\it Proof.} As in Lemma \ref{curve}: the first and last statements are immediate from the construction ($K_{m_i}\subset M\cap Q_{m_i}$ by Theorem A, and every other point of $\mathcal K$ lies in $\mathcal S\subset M$ by Lemma \ref{curve}); the Hausdorff estimate follows because $K_m$ and $\mathcal K$ differ only inside the union of the cubes $Q_{m'_i}$ of stage $m'>m$, each of diameter $\sqrt3\cdot3^{-(m+1)}$. $\square$

\begin{lem}\label{wild}
Let ${\mathcal{K}}$ be constructed as above, with each knot $K_{m_i}$ nontrivial except for finitely many indices. Then ${\mathcal{K}}$ is a wild knot. Moreover, ${\mathcal{K}}$ is wild at each point of the usual Cantor set ${\mathcal{C}}$ contained in the unit segment $e=\{(1,t,0)\,\,:\,\,t\in [0,1]\}\subset M$, and it is locally flat at every other point of $\mathcal K$.
\end{lem}

\noindent{\it Proof.} (Compare \cite{gaby}, \cite{HVD}, and \cite{DH}.)

\smallskip
\noindent\emph{Wildness at points of $\mathcal C$.} Fix $x\in\mathcal C$ and let $U$ be any open ball centered at $x$. Because the diameters $\mathrm{diam}(T_{m_i})=3^{-(m+1)}$ tend to $0$ and the arcs $T_{m_i}$ accumulate exactly on $\mathcal C$, there is $m_0$ such that $U\cap\mathcal K$ contains $\gamma_{m_i}(T_{m_i})$ for every $(m,i)$ with $m\geq m_0$ and $T_{m_i}\cap U\ne\emptyset$; by the self-similar structure of the construction, infinitely many such pairs exist for every $U$, and by hypothesis all but finitely many of the corresponding $K_{m_i}$ are nontrivial knots.\\


\noindent We will compute $\pi_1(U\setminus \mathcal K)$ directly, from the actual embedded set $\mathcal K$ constructed above, using a locally finite van Kampen decomposition. Write $N_j=\gamma_{m_j i_j}(T_{m_j i_j})$ for the corresponding knotted arc and $B_j$ for a small closed ball neighborhood of $N_j$ in $U$, chosen small enough that the $B_j$ are pairwise disjoint, disjoint from $\mathcal C$, and meet $\mathcal K$ only in $N_j$ (possible since the $N_j$ are pairwise disjoint compacta accumulating only on $\mathcal C$).
List the pairs $(m,i)$ with $\gamma_{m_i}(T_{m_i})\subset U$ and $K_{m_i}$ nontrivial as $(m_1,i_1),(m_2,i_2),\ldots$ (countably many, by the previous paragraph, and ordered so that $\mathrm{diam}\,B_j\to0$). For $n\geq1$ define
$$
W_n:=U\setminus\Big(N_1\cup\cdots\cup N_n\ \cup\ \bigcup_{j>n}B_j\Big),
$$
\emph{i.e.}, $W_n$ is $U$ with the $n$ balls $B_1,\ldots,B_n$ \emph{drilled} (only the knotted arc $N_j$ removed from each) and every other ball $B_j$ ($j>n$) removed \emph{solid}. Since only finitely many balls are drilled at stage $n$ and the solid balls get un-plugged one at a time as $n$ increases, $\{W_n\}$ is an increasing sequence of open subsets of $U\setminus\mathcal K$ with $\bigcup_nW_n=U\setminus\mathcal K$ (a point of $U\setminus\mathcal K$ either avoids every $B_j$, in which case it lies in every $W_n$, or lies in some fixed $B_{j_0}\setminus N_{j_0}$, in which case it lies in $W_n$ for all $n\geq j_0$).\\

\noindent By van Kampen's theorem, applied once for each of the $n$ drilled balls,
$$
\pi_1(W_n)\;\cong\; H *_{\langle\mu_1\rangle} G_1 *_{\langle\mu_2\rangle}G_2*_{\langle\mu_3\rangle}\cdots *_{\langle\mu_n\rangle} G_n,
$$
where: $H:=\pi_1\big(U\setminus\bigcup_jB_j\big)$ is the fundamental group of $U$ with \emph{every} ball $B_j$ removed solid -- a single fixed space, not depending on $n$ or on the knot types $K_{m_ji_j}$ at all, since none of the $N_j$ are exposed in it; $G_j:=\pi_1(B_j\setminus N_j)$ is the knot group of $K_{m_ji_j}$ (since $B_j$ is a $3$-ball meeting $\mathcal K$ only in the properly embedded, boundary-unknotted arc $N_j$, the pair $(B_j,N_j)$ is a knotted ball-arc pair of the same type as $(\mathbb S^3,K_{m_ji_j})$, so $\pi_1(B_j\setminus N_j)\cong\pi_1(\mathbb S^3\setminus K_{m_ji_j})$); and $\mu_j\in G_j$ is a meridian of $N_j$, identified with the corresponding generator of $\pi_1(\partial B_j\setminus N_j)\cong\mathbb Z$ (note $\partial B_j\setminus N_j$ is a $2$-sphere minus two points, homotopy equivalent to a circle -- the meridian -- which is the entire gluing locus in this application of van Kampen; unlike the connected-sum-of-two-knots gluing along a peripheral torus, here only one side, $B_j\setminus N_j$, contributes a nontrivial knot group, and the outer piece $H$ sees only the meridian circle, not a longitude).\\

\noindent Since $H$ is the same group for every $n$, the sequence of amalgams $A_n:=H*_{\langle\mu_1\rangle}G_1*\cdots*_{\langle\mu_n\rangle}G_n\cong\pi_1(W_n)$ is genuinely increasing, $A_n\subset A_{n+1}$ (via the inclusion-induced map that is the identity on $H,G_1,\ldots,G_n$ and includes $\langle\mu_{n+1}\rangle\hookrightarrow G_{n+1}$), and $U\setminus\mathcal K=\bigcup_nW_n$. By the continuity property of $\pi_1$ under increasing unions of open sets (a standard consequence of van Kampen's theorem for directed unions, applicable here because each loop and each homotopy in $U\setminus\mathcal K$ is compact and hence lies in some $W_n$),
$$
\pi_1(U\setminus\mathcal K)\;\cong\;\varinjlim_n \pi_1(W_n)\;\cong\; H *_{\langle\mu_1\rangle}G_1*_{\langle\mu_2\rangle}G_2*_{\langle\mu_3\rangle}\cdots
$$
an iterated amalgamated free product of infinitely many knot groups $G_j$, each amalgamated over its own meridian subgroup $\langle\mu_j\rangle\cong\mathbb Z$, with a fixed base group $H$.\\

\smallskip
\noindent\emph{This group is infinitely generated.} Since $K_{m_ji_j}$ is a nontrivial knot, $G_j$ is a nonabelian group (the group of a nontrivial knot is never abelian, e.g.\ because its abelianization is always $\mathbb Z$ but the group itself has a nontrivial commutator subgroup, or simply because $\mathbb S^3\setminus K_{m_ji_j}$ is not a $K(\mathbb Z,1)$ for the trivial reason that $\mathbb S^3\setminus\mathrm{unknot}\simeq \mathbb S^1$ is the only knot complement homotopy equivalent to a circle, by the classical theorem that the unknot is the only knot with abelian, equivalently infinite cyclic, group); in particular $\langle\mu_j\rangle\subsetneq G_j$ is a proper subgroup. Recall $A_n\cong\pi_1(W_n)$ from above, so that $\pi_1(U\setminus\mathcal K)=\bigcup_n A_n$ as an increasing union of subgroups (a ``graph of groups'' along a line, \emph{i.e.}, an iterated free product with amalgamation, which is the fundamental group of the graph of groups consisting of the vertex groups $H, G_1, G_2,\ldots$ joined in a line by edge groups $\langle\mu_j\rangle$). By the normal form theorem for amalgamated free products (Bass--Serre theory), an element of $G_{n+1}\setminus\langle\mu_{n+1}\rangle$ is never conjugate, in $A_{n+1}$, into $A_n$; in particular $A_{n+1}\ne A_n$, so the union $\bigcup_n A_n$ is a strictly increasing union of subgroups. If $\pi_1(U\setminus\mathcal K)$ were finitely generated by $g_1,\ldots,g_k$, then, since each $g_r$ lies in $A_{n_r}$ for some $n_r$ (every element of the direct limit lies in some finite stage), all $g_r$ would lie in $A_N$ for $N=\max_r n_r$, forcing $\pi_1(U\setminus\mathcal K)=A_N$, contradicting that $A_{N+1}\supsetneq A_N$. Hence $\pi_1(U\setminus\mathcal K)$ is infinitely generated.\\

\noindent Since $U$ was an arbitrary open ball centered at $x$, no neighborhood of $x$ has simply connected (or even finitely-generated $\pi_1$) punctured complement in $\mathcal K$; in particular $x$ is not locally flat, \emph{i.e.}, $x$ is a wild point.\\

\smallskip
\noindent\emph{Local flatness away from $\mathcal C$.} Let $x\in\mathcal K\setminus\mathcal C$. Since the notches and insertions of $\mathcal K$ accumulate only on $\mathcal C$ (Sections 3.1--3.2), $x$ has a neighborhood meeting only finitely many pieces of $\mathcal K$, so $\mathcal K$ coincides with a finite PL arc near $x$, which is locally flat there (this is exactly Case 1 of the proof of Theorem 1 below, restated here since it does not depend on the density of nontrivial summands and so applies verbatim under the present hypothesis). Under the hypothesis of this lemma (all but finitely many $K_{m_i}$ nontrivial), \emph{every} point of $\mathcal C$ is wild by the argument above, so there are no further points to examine; the more delicate case of points in $\mathcal C$ that remain tame because only finitely many nontrivial summands accumulate there is precisely the setting of Theorem 1 below (Case 2 of its proof), where the hypothesis on $\mathcal K$ is different (only finitely many, not all but finitely many, summands are nontrivial near a given tame point). $\square$\\

\noindent We are now ready to prove our first main result.

\begin{main1}\label{main1-restated}
For every finite subset $T=\{p_1,\ldots,p_n\}$ of the standard Cantor set $\mathcal{C}$ contained in the edge $e\subset M$ described in Section~3.2, there is an explicit, recursively constructed wild knot $\mathcal{K}_T\subset M$ whose set of wild points is exactly $T$. Consequently, the Menger sponge $M$ contains infinitely many pairwise non-equivalent, explicitly constructed wild knots, distinguished by the cardinality of their wild point sets.
\end{main1}

\noindent{\it Proof.} Let $T=\{p_1,\,p_2,\,\ldots,\,p_n\}$ be a finite subset of the Cantor set ${\mathcal{C}}$ contained in the edge $e\subset M$. We construct a wild knot in the Menger sponge whose set of wild points is exactly $T$. Doing this for each $n\in\mathbb{N}$ yields wild knots with $n$ wild points. Since equivalent wild knots must have the same number of wild points, these knots are pairwise nonequivalent.\\

\noindent Consider the knot $\mathcal{K}=\mathcal K_T$ constructed above (Lemma \ref{Kembed}). We now specify the knots $K_{m_j}$ at each stage. For each $p_i\in T$, $i=1,2,\ldots,n$, by construction there exists a sequence of segments $S_m(p_i)\subset e$ such that $p_i=\bigcap_m S_m(p_i)$, where $S_m(p_i)$ belongs to the $m$-th stage of the Cantor set ${\mathcal{C}}\subset e$. For each $m\in\mathbb{N}$ and $i=1,\ldots,n$, we require that the knot $K_m(p_i)$ located on the segment of $\mathcal{S}$ that meets $S_m(p_i)$ at an endpoint be non-trivial. All other $K_{m_s}$ in $\mathcal{K}$ are taken to be the \emph{unknot} (not merely ``trivial,'' but literally the straight, unknotted connect-sum summand -- \emph{i.e.}, we simply do not modify $e_{m_s}$ at those indices, so $K_{m_s}$ contributes nothing beyond a straight segment) (see Figure \ref{W2}).\\

\noindent {\bf Wildness at $T$.} Let $p_i\in T$ and let $U$ be an open neighborhood of $p_i$. By construction, infinitely many nontrivial knots $K_m(p_i)$ lie in $U$, and the argument in the proof of Lemma \ref{wild} (which used only that infinitely many nontrivial summands accumulate at the point, not that \emph{every} summand near the point is nontrivial) shows that $\pi_1(U\setminus\mathcal{K})$ is infinitely generated; hence $p_i$ is a wild point.\\

\smallskip
\noindent {\bf Local flatness away from $T$.} Let $x\in\mathcal{K}\setminus T$. We exhibit an explicit ambient homeomorphism straightening $\mathcal K$ near $x$, which is the local flatness proof.

\noindent \emph{Case 1: $x\notin\mathcal C$.} Then $x$ has a neighborhood meeting only finitely many of the pieces of $\mathcal K$ (the notches and insertions accumulate only on $\mathcal C$), so $\mathcal K$ coincides with a finite polygonal (PL) arc near $x$. A PL arc is locally flat at each of its points (its complement is locally simply connected there, by a direct local computation, or by Bing's theorem that locally tame sets are tame \cite{Bi}); hence $x$ is locally flat.\\

\noindent \emph{Case 2: $x\in\mathcal C\setminus T$.} Since $x\notin T$, by construction there is $m_x$ such that every $K_{m_i}$ with $T_{m_i}$ meeting a small enough neighborhood of $x$ and $m\geq m_x$ is unknot -- precisely, choose $m_x$ so large that the unique branch $S_m(x)$ of the Cantor set containing $x$ at stage $m\geq m_x$ contains no $S_m(p_i)$, $i=1,\ldots,n$ (possible since $x\ne p_i$ for all $i$, so the branches separate at some finite stage). Let $U_x$ be the corresponding neighborhood of $x$: the union of the (shrinking, nested) segments $S_m(x)$, $m\geq m_x$, together with the notches and insertions built on them.\\

\noindent For $m\geq m_x$, every insertion $K_{m_i}$ with $T_{m_i}\subset U_x$ is, by hypothesis, the straight (unknotted) segment $e_{m_i}$ itself -- \emph{i.e.}, \emph{no geometric modification is made at all} at those stages within $U_x$ beyond the (already tame, PL) squareflake notches of Section 3.1. Hence $\mathcal K\cap U_x=\mathcal S\cap U_x$: within $U_x$, the wild-knot construction contributes nothing beyond the squareflake curve. It therefore suffices to show that $\mathcal S$ itself is locally flat at every point of $\mathcal C$, which we now do, once and for all (this also completes the deferred part of Lemma \ref{wild}).\\

\smallskip
\noindent {\bf Lemma (local flatness of $\mathcal S$ along $\mathcal C$).} \emph{Every point $x\in\mathcal C\subset\mathcal S$ is a locally flat point of $\mathcal S$.}\\

\noindent{\it Proof.} Fix $x\in\mathcal C$. For $m\geq1$ let $I_m(x)=[a_m,b_m]$ be the level-$m$ triadic interval containing $x$ (so $b_m-a_m=3^{-m}$, $I_{m+1}(x)\subset I_m(x)$, $\bigcap_m I_m(x)=\{x\}$). At most one middle third of $I_m(x)$ is ever removed at stage $m+1$ (namely the middle third of $I_m(x)$ itself, if it happens to be a $T_{m+1,k}$-type interval; whether or not it is, $I_{m+1}(x)$ is one of the two outer thirds of $I_m(x)$, hence disjoint from the middle third notch built at stage $m+1$, if any, on $I_m(x)$). Consequently: \emph{every notch of $\mathcal S$ that meets $I_m(x)$, other than possibly one built on $I_m(x)$'s own middle third, is entirely disjoint from $I_{m}(x)$}. The only notches meeting $I_m(x)$ at all are (i) at most one notch of diameter $\leq \sqrt2\cdot3^{-(m+1)}\ll 3^{-m}$ sitting inside $I_m(x)\setminus I_{m+1}(x)$, and (ii) the (recursively, arbitrarily small) notches inside $I_{m+1}(x)$ itself.\\

\noindent This gives a nested sequence of balls $B_m=B(x,2\cdot3^{-m})\supset B_{m+1}$ such that $\mathcal S\cap B_m$ consists of: a single straight core arc through $x$ of diameter $O(3^{-m})$ (coming from $I_{m}(x)\setminus I_{m+1}(x)$, translated to $\phi_0$-coordinates), together with at most one attached notch of diameter $O(3^{-(m+1)})$, together with $\mathcal S\cap B_{m+1}$ (the same picture one level deeper). In particular, there is an ambient isotopy $h_m$ of $\mathbb R^3$, supported in the shell $B_m\setminus B_{m+1}$, straightening the one possible notch in that shell to the straight segment $I_m(x)\setminus I_{m+1}(x)$, and fixing $\partial B_m\cup\partial B_{m+1}$ pointwise, with $\|h_m-\mathrm{id}\|_\infty\leq \mathrm{diam}(B_m)=4\cdot3^{-m}$ (such an isotopy exists because the notch and the straight segment are two unknotted, disjoint-except-at-endpoints PL arcs in the shell sharing the same two endpoints on $\partial B_m\cup\partial B_{m+1}$; two such arcs always cobound an embedded PL disk in the shell -- here, visibly, the small planar rectangle swept out between the three-segment notch and the straight segment -- and pushing across a collar of that disk gives an ambient isotopy, supported near the disk and hence in the shell, carrying one arc to the other and fixing the shell's boundary spheres). Composing $h_m$ over all $m\geq m_0$, in order of decreasing shell radius, gives a well-defined map
$$
h=\lim_{m\to\infty}\big(h_{m_0}\circ h_{m_0+1}\circ\cdots\circ h_m\big)
$$
because the supports $B_m\setminus B_{m+1}$ are pairwise disjoint and shrink to $\{x\}$ (so the infinite composition is, at every point other than $x$, eventually constant, and at $x$ itself both sides converge to $x$); $h$ is continuous, with continuous inverse built the same way from $h_m^{-1}$, so $h$ is a homeomorphism of $\mathbb R^3$ supported in $B_{m_0}$, fixing $x$, and $h(\mathcal S\cap B_{m_0})$ is the straight segment $I_{m_0}(x)$ (in $\phi_0$-coordinates). Hence $x$ is a locally flat point of $\mathcal S$. $\square$\\

\noindent Returning to Case 2: since $\mathcal K\cap U_x=\mathcal S\cap U_x$ and $\mathcal S$ is locally flat at $x\in\mathcal C$ by the lemma just proved, $\mathcal K$ is locally flat at $x$ as well.

\smallskip
\noindent Combining Cases 1 and 2, every $x\in\mathcal K\setminus T$ is locally flat, and every $p_i\in T$ is wild. Therefore, the set of wild points of $\mathcal K$ is exactly $T$. This completes the proof of Theorem 1. $\square$\\

\subsection{Wild knots of dynamically defined type in the Menger sponge}

\noindent Before proving our next result, we recall the construction of wild knots of dynamically defined type; for details, see \cite{DH}.\\

\noindent Let $K$ be a tame knot. An $n$-beaded necklace $T^{\circ}$ subordinate to $K$ consists of the union of $n$ disjoint closed round 3‑balls (called pearls) $B_{1}:=B_{r_1}(c_1)$, $B_{2}:=B_{r_2}(c_2),\ldots , B_{n}:=B_{r_n}(c_n)$ in $\mathbb{S}^{3}$ (so $B_{i}\cap B_{j}=\emptyset$ for $i\neq j$), such that $c_i\in K$ and the segment of $K\cap B_i$ is unknotted for each $i$. An $n$-pearl chain necklace $T$ is the union $T=T^{\circ}\cup K$ (see Figure \ref{T2}).\\

\begin{figure}[h] 
\begin{center}
\includegraphics[height=3cm]{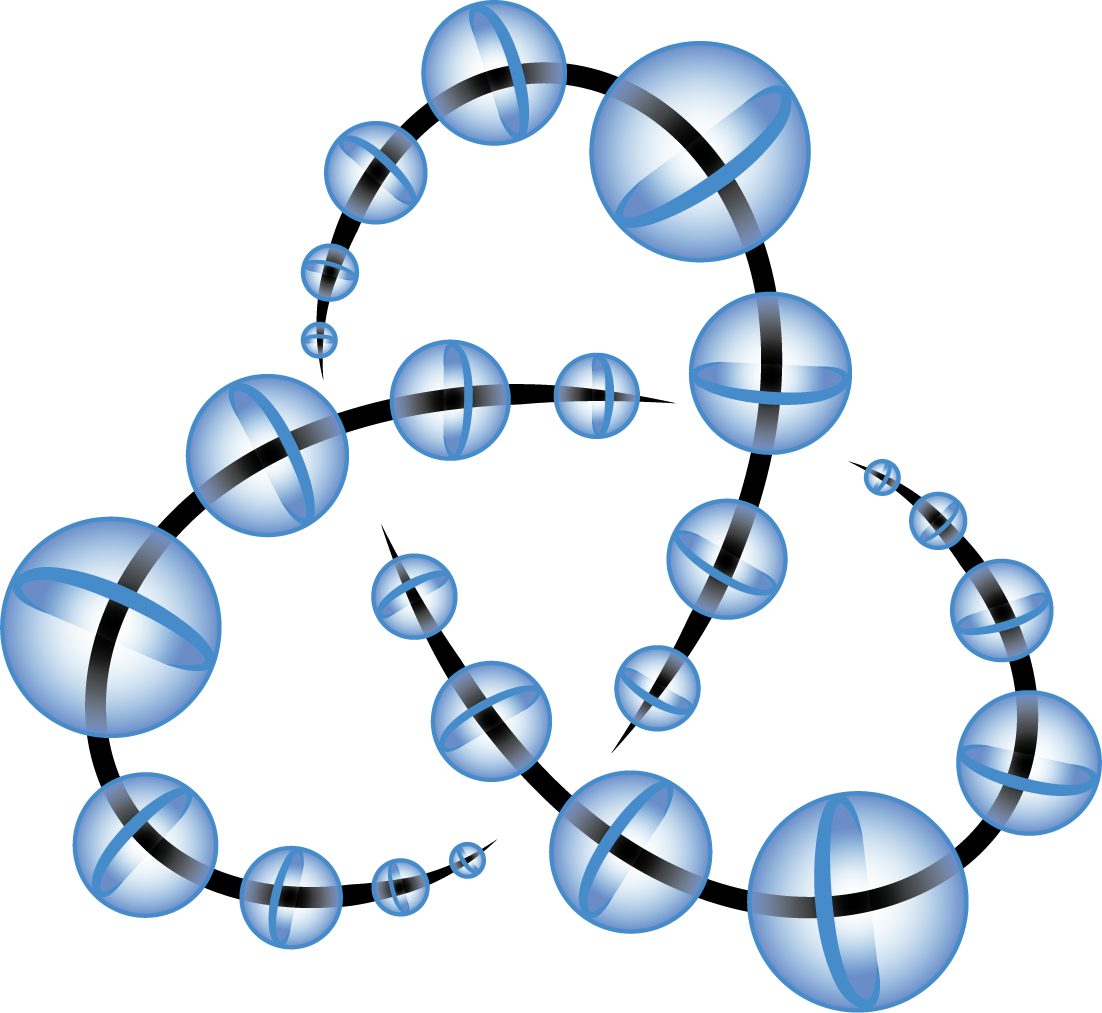}
\end{center}
\caption{\sl A pearl chain necklace.} 
\label{T2}
\end{figure}

\noindent Let $\Gamma_{T^{\circ}}$ be the group generated by reflections $I_{j}$ through $\Sigma_{j}=\partial B_j$ ($j=1,\ldots,n$). Then $\Gamma_{T^{\circ}}$ is a discrete subgroup of $M\ddot{o}b(\mathbb{S}^{3})$ whose limit set is a Cantor set (\cite{kap1}, \cite{maskit}). To construct a wild knot, we build a nested sequence of pearl chain necklaces via the action of $\Gamma_{T^{\circ}}$ and take the inverse limit $\Lambda(K,T^{\circ})$.\\

\noindent Notice that reflecting with respect to each $\Sigma_j$ maps both a mirror image of $K$, denoted $\bar{K}$, and the corresponding pearl chain strand $T-B_j$ into the ball $B_j$. Hence $B_j$ contains a pearl chain strand $\tau_{1_j}=I_j(T-B_j)$ such that the pair $(B_j,\tau_{1_j})$ is homeomorphic to $(C,\bar{K})$, where $C$ is a solid cylinder and $\bar{K}$ is the mirror image of $K$. Thus we obtain a new pearl chain necklace $T_{1_j}=(T-B_j)\cup \tau_{1_j}$, which is obtained from $T$ by replacing the pearl $B_j$ by $\tau^1_j$. This pearl necklace is subordinate to a knot ${\mathrm{K}}_{1_j}$ isotopic to the connected sum of $K$ and its mirror image $\bar{K}$. Similarly we obtain the corresponding beaded strand $\tau^{\circ}_{1_j}=I_j(T^{\circ}-B_j)$ and the beaded necklace $T^{\circ}_{1_j}=(T^{\circ}-B_j)\cup \tau^{\circ}_{1_j}$ (see Figure~\ref{T3}).\\

\begin{figure}[h]  
\begin{center}
\includegraphics[height=4cm]{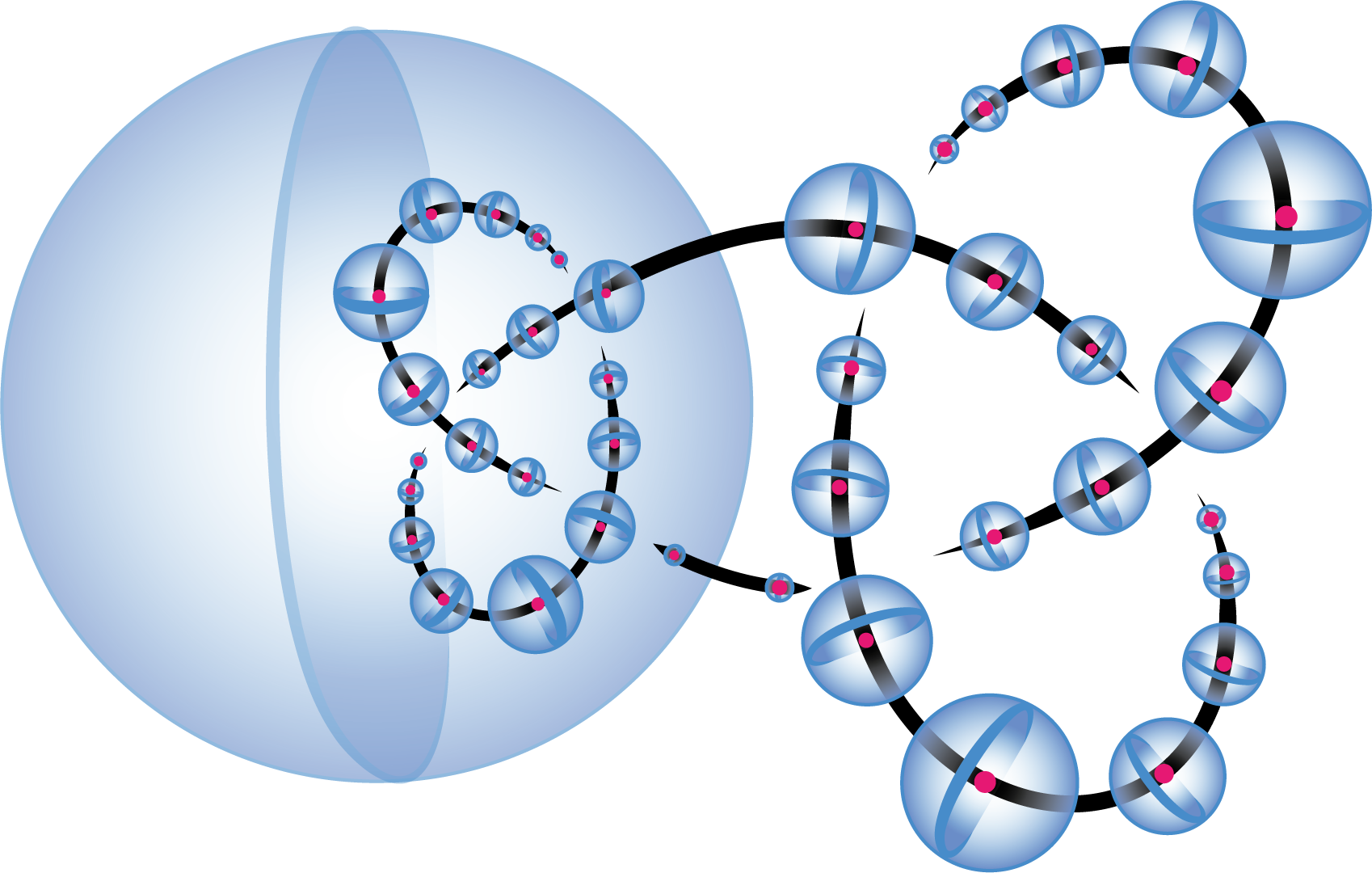}
\end{center}
\caption{\sl A schematic picture of a pearl chain necklace after a reflection.} 
\label{T3}
\end{figure}
 
\noindent After reflecting with respect to each $\Sigma_j$, we obtain a new beaded necklace $T^{\circ}_{1}$ consisting of the union of $l_1=n(n-1)$ pearls $B_j^1$, $j\in\{1,\dots, l_1\}$, subordinate to a new knot ${\mathrm{K}}_{1}$ which is isotopic to the connected sum of $K$ and $n$ copies of its mirror image $\bar{K}$. Let $T_1=T^{\circ}_{1}\cup {\mathrm{K}}_{1}$ be the corresponding pearl chain necklace; then $T^{\circ}_{1}\subset T^{\circ}_{0}=T^{\circ}$ and $T_1\subset T=T_0$.\\

\noindent We continue this process. At the $m$‑th stage we obtain a new necklace $T^{\circ}_{m}$ which is the union of $l_m=n(n-1)^{m}$ pearls $B_j^m$, $j\in\{1,\dots, l_m\}$, subordinate to a polygonal knot ${\mathrm{K}}_{m}$. Let $T_m=T^{\circ}_{m}\cup {\mathrm{K}}_{m}$ be the corresponding pearl chain necklace; by construction $T^{\circ}_{m}\subset T^{\circ}_{m-1}$ and $T_{m}\subset T_{m-1}$. Each pearl $B_j^{m-1}$ ($j\in\{1,\dots, l_{m-1}\}$) contains a pearl chain strand $\tau_{m_j}$ (a thickened one‑strand non‑trivial tangle) of the same knot type as $K$ or $\bar{K}$, together with a beaded strand $\tau^{\circ}_{m_j}$ that is the union of $n-1$ disjoint pearls subordinate to the corresponding knotted arc. The union of all strands $\tau^{\circ}_{m_j}$ is $T^{\circ}_{m}$.\\

\noindent The inverse limit space
$$
\Lambda(K,T^{\circ})=\varprojlim_{m} T_{m}=\bigcap_{m=0}^{\infty} T_{m}
$$
is a wild knot, called a wild knot of dynamically defined type.

\begin{rem}[Pearls shrink to points, reconciling $\varprojlim$ with $\varinjlim$]\label{shrinkpearls}
Since $\Gamma_{T^{\circ}}$ is generated by reflections through $n\geq2$ pairwise disjoint round spheres, the standard ping-pong (Schottky-type) contraction estimate for such groups (\cite{maskit}, \cite[\S 2]{kap1}) gives constants $C>0$ and $0<\lambda<1$, depending only on $T^{\circ}$, such that every pearl $B_j^m$ occurring at the $m$-th stage satisfies $\mathrm{diam}(B_j^m)\leq C\lambda^{m}$. Consequently $\max_j\mathrm{diam}(B_j^m)\to0$ as $m\to\infty$, and the nested compacta $T_m$ Hausdorff-converge to $\Lambda(K,T^{\circ})=\bigcap_m T_m$. In particular, if $\mathrm{K}_m\subset T_m$ denotes the polygonal ``skeleton'' knot subordinate to $T_m$ (i.e. $T_m$ with the pearls forgotten), then $\mathrm{K}_m\to\Lambda(K,T^{\circ})$ in the Hausdorff metric as well, since $T_m$ is a regular neighborhood of $\mathrm{K}_m$ of radius $\leq C\lambda^{m}\to0$. This is the precise sense in which the a priori \emph{decreasing}, nested-intersection definition $\Lambda(K,T^\circ)=\bigcap_m T_m$ agrees with the \emph{increasing}, embedded-limit description $\Lambda(K,T^\circ)=\lim_m \mathrm K_m$ used informally below (and in \cite{gaby}, \cite{GA}): both refer to the same Hausdorff limit, and Lemma \ref{embed} applies to realize it as an honest embedding $\mathbb S^1\to\mathbb R^3$, exactly as it did for $\mathcal S$ and $\mathcal K$ in Sections 3.1--3.2, since the successive modifications $\mathrm{K}_{m-1}\rightsquigarrow \mathrm{K}_m$ are supported in the pearls $B_j^{m-1}$, which are pairwise disjoint and of diameter $\leq C\lambda^{m-1}\to0$.
\end{rem}

\begin{main2}\label{main2-restated}
Any wild knot of dynamically defined type is isotopic to a wild knot contained in the Menger sponge such that its set of wild points is contained in a tame Cantor set contained in $M$.
\end{main2}

\noindent{\it Proof.} Let $K$ be a tame knot. Let $T$ be an $n$-pearl chain necklace subordinate to $K$ ($n\geq3$ arbitrary), and let $\Lambda(K,T^{\circ})$ be the corresponding wild knot of dynamically defined type. We show directly, for the given $n$, that $\Lambda(K,T^{\circ})$ is isotopic to a knot $\mathcal{K}$ contained in the Menger sponge. \\


\noindent Consider the squareflake curve $\mathcal{S}$. Let $T=\{0\}\times [\frac{1}{3},\frac{2}{3}]\times \{0\}$ be the middle‑third segment of the left edge 
$(0,0,0)$-$(0,1,0)$, and let $Q$ be a cube belonging to the second stage $M_2$ of the Menger sponge such that $T\subset Q$. Notice that $T$ lies entirely on $\mathcal{S}$. The cube $Q$ is homothetically equivalent to $M$; by the argument used in Theorem A, there exists an isotopic copy of our knot $K$, say $K'$, contained in $M\cap Q$, with an unknotted segment $e$ of $K'$ lying in $T$ (see Figure \ref{knotMS}). This allows us to define $K_{0}$ as the connected sum of $\mathcal{S}$ with $K'$ along $e$:
$$
K_0\cong {\mathcal{S}}\#K'.   
$$
By construction $K_0$ lies in the Menger sponge and is isotopic to $K$ via an ambient isotopy $H_0:\mathbb{R}^3\times [0,1]\rightarrow\mathbb{R}^3$.\\

\begin{figure}[h] 
\begin{center}
\includegraphics[height=8cm]{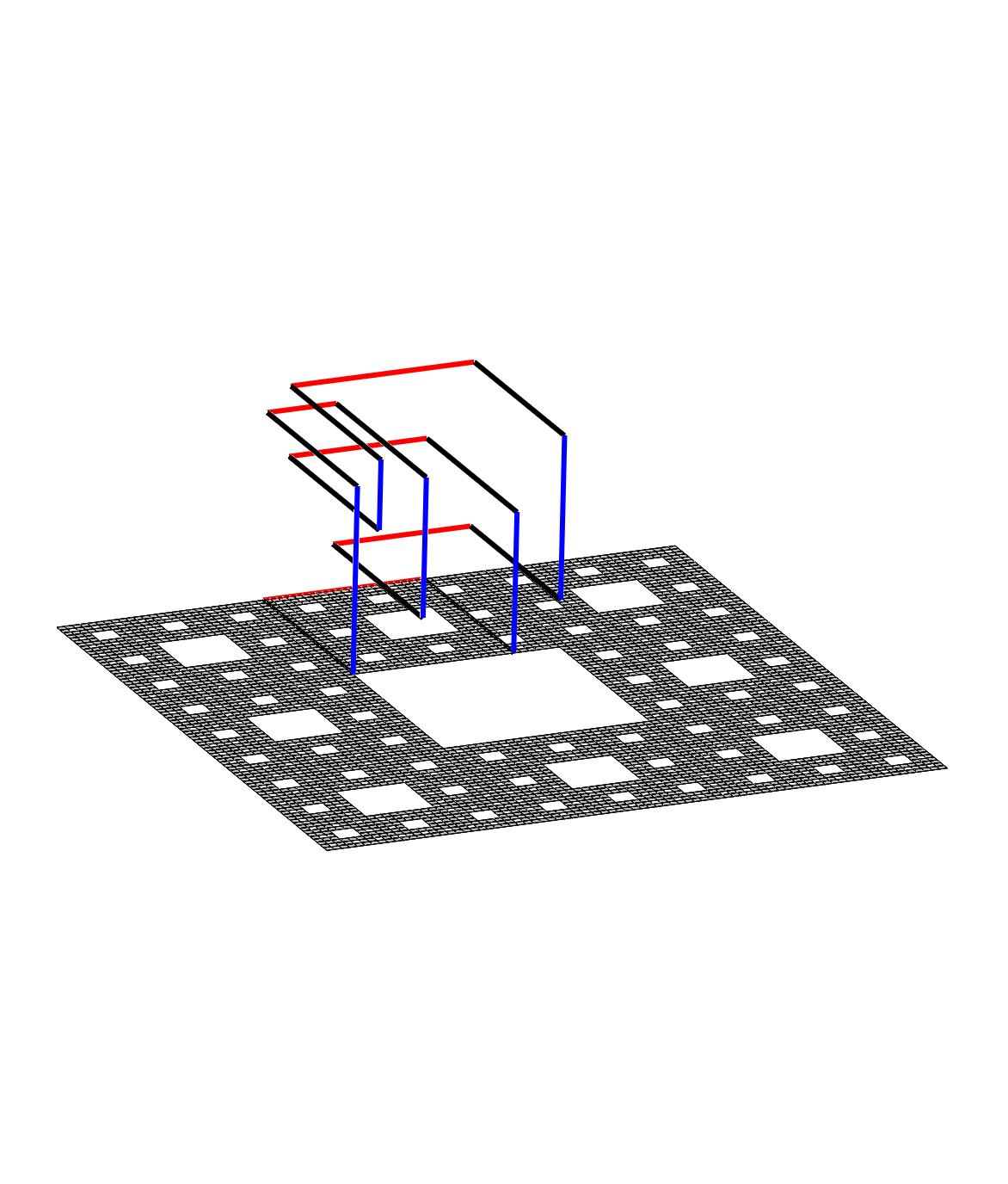}
\end{center}
\caption{\sl The figure-eight knot on $M\cap Q$.} 
\label{knotMS}
\end{figure}. 

\noindent We now compare the two constructions stage by stage, matching the branching combinatorics of $\Gamma_{T^\circ}$ (recalled above: $l_0:=n$ pearls initially, and $l_m=n(n-1)^m$ pearls at stage $m$, so that $l_{m-1}$ new knotted summands are introduced at stage $m$, for every $n\geq2$) using the insertion device of Remark \ref{differentknots}(2) with a \emph{stage-dependent} number of insertions, rather than the fixed count ``$3,6,12,\ldots$'' of Section 3.2 (which is tied to the specific doubling combinatorics of the squareflake curve and only matches $\Gamma_{T^\circ}$ when $n=3$).\\

\noindent {\it First stage.} On the segment $T\subset\mathcal S$ used to build $K_0$, apply Remark \ref{differentknots}(2) with parameter $n$: subdivide $T$ into $2n+1$ subsegments and choose $n$ non-consecutive ones, $T_{1_1},\ldots,T_{1_n}$, each lying, as a side, on a cube $Q_{1_i}$ of the corresponding stage of $M$. By Theorem A there are isotopic copies $K_{1_i}$ of $\bar K$, $i=1,\ldots,n$, contained in $M\cap Q_{1_i}$, pairwise disjoint, each with an unknotted marked segment $e_{1_i}\subset T_{1_i}$; we define
$$
K_1\cong K_0\#K_{1_1}\#\cdots\#K_{1_n}.
$$
Thus $K_1$ is isotopic to the connected sum of $K$ with $n$ copies of $\bar K$, matching the necklace's $\mathrm K_1\cong K\#(n\text{ copies of }\bar K)$ exactly, and $K_1$ lies in the Menger sponge. On each of the $n$ new summands $K_{1_i}$ we further mark $n-1$ non-consecutive subsegments $T_{2_{i,1}},\ldots,T_{2_{i,n-1}}$ (again via Remark \ref{differentknots}(2), now with parameter $n-1$, applied to a middle-third segment of $K_{1_i}$'s own defining cube), for a total of $n(n-1)=l_1$ marked segments available for the second stage.\\

\noindent {\it $m$-th stage ($m\geq2$).} Inductively, suppose $K_{m-1}$ has been built together with $l_{m-1}=n(n-1)^{m-1}$ pairwise disjoint marked segments $T_{m_1},\ldots,T_{m_{l_{m-1}}}$, one on each of the $l_{m-1}$ summands introduced at stage $m-1$ (with, if $m\geq3$, $n-1$ marked segments per summand of the previous stage, by induction). By Theorem A, insert an isotopic copy $K_{m_i}$ of $K$ or $\bar K$ (matching the parity/orientation convention of the necklace construction: the mirror image is applied at each successive reflection) into $M\cap Q_{m_i}$ along each $T_{m_i}$, $i=1,\ldots,l_{m-1}$, pairwise disjoint, defining
$$
K_m\cong K_{m-1}\#K_{m_1}\#\cdots\#K_{m_{l_{m-1}}},
$$
a knot lying in the Menger sponge and isotopic to the same connected sum of copies of $K$ and $\bar K$ as the necklace's $\mathrm K_m$. On each of the $l_{m-1}$ new summands, mark $n-1$ further non-consecutive subsegments via Remark \ref{differentknots}(2) with parameter $n-1$, giving $l_{m-1}(n-1)=l_m$ marked segments for stage $m+1$.\\

\noindent {\bf The embedded limit and the isotopy.} By construction, $K_{m-1}$ and $K_m$ agree outside the union of the $l_{m-1}$ cubes $Q_{m_i}$ used at stage $m$, and (exactly as in Sections 3.1--3.2) once a segment is used for an insertion it is never modified again. As in Lemma \ref{embed} and Lemma \ref{Kembed}, this produces an honest embedded limit $\mathcal K=\Phi(\mathbb S^1)\subset M$, $\mathcal K=\lim_m K_m$ in the Hausdorff metric, \emph{provided} $\mathrm{diam}(Q_{m_i})\to0$ as $m\to\infty$ uniformly in $i$ -- which holds because $Q_{m_i}$ is a cube of a Menger-sponge stage whose depth increases with $m$ by construction (each round of Remark \ref{differentknots}(2) passes to a further stage of $M$). By Remark \ref{shrinkpearls}, the corresponding skeleton knots $\mathrm K_m$ subordinate to $T_m$ satisfy $\mathrm K_m\to\Lambda(K,T^\circ)$ in the Hausdorff metric as well, at the geometric rate $C\lambda^m$.

\smallskip
\noindent By construction there are, for each $m$, ambient isotopies $H_m:\mathbb R^3\times[0,1]\to\mathbb R^3$ with $H_m(K_{m-1},1)=K_m$-side data matched to $\mathrm K_{m-1}\to\mathrm K_m$ (both are obtained from the previous stage by the same connected-sum operation along corresponding segments, so the identification $K_m\leftrightarrow \mathrm K_m$ extends to an isotopy of a neighborhood of the $m$-th stage summands), and, crucially, we choose $H_m$ to be the identity outside a small regular neighborhood $N_m$ of $\bigcup_i(Q_{m_i}\cup B_j^{m-1})$ (the union of the stage-$m$ Menger cubes and the corresponding stage-$(m-1)$ pearls), which has $\mathrm{diam}(N_m)\to0$ and whose components are pairwise disjoint for different $m$ once $m$ is large (since both the $Q_{m_i}$ and the $B_j^{m-1}$ shrink to the finitely many, then countably many, points of the respective limit Cantor sets). This support control -- not available in the original argument, where $H=\lim H_m$ was asserted without such control -- is exactly what is needed:

\begin{cla}\label{isotopylimit}
The composition $H_t:=\cdots\circ (H_m)_t\circ\cdots\circ(H_2)_t\circ(H_1)_t$ (well defined for each fixed $t\in[0,1]$ because the supports $N_m$ are, after discarding finitely many exceptions and shrinking each $H_m$'s support if necessary, pairwise disjoint and accumulate only on the common limit Cantor set $\mathcal C_\Lambda:=\bigcap_m\bigcup_i B_j^m$) defines an ambient isotopy $H:\mathbb R^3\times[0,1]\to\mathbb R^3$, i.e.\ each $H_t$ is a homeomorphism depending continuously on $t$, with $H_0=\mathrm{id}$.
\end{cla}

\noindent{\it Proof of Claim.} Fix $t$. For $x\notin\mathcal C_\Lambda$, $x$ has a neighborhood meeting only finitely many $N_m$ (since $\mathrm{diam}(N_m)\to0$ and the $N_m$ accumulate only on $\mathcal C_\Lambda$), so $H_t$ agrees, near $x$, with a finite composition of homeomorphisms and is therefore a homeomorphism near $x$. For $x\in\mathcal C_\Lambda$: given $\varepsilon>0$, choose $M_0$ with $\mathrm{diam}(N_m)<\varepsilon$ for all $m\geq M_0$; then for $y$ within $\varepsilon$ of $x$, all the isotopies $H_m$, $m\geq M_0$, that could possibly move $y$ do so by less than $\mathrm{diam}(N_m)<\varepsilon$ (they are supported in balls of that diameter), so $|H_t(y)-y|$ receives contributions only from finitely many $H_m$ ($m<M_0$, each individually continuous) plus a tail bounded by $\sum_{m\geq M_0}\mathrm{diam}(N_m)$, which is small if the $N_m$ are chosen (as we may, shrinking $H_m$'s support if necessary) with $\sum_m\mathrm{diam}(N_m)<\infty$ (possible since $\mathrm{diam}(N_m)=O(\lambda^m)$ geometrically by Remark \ref{shrinkpearls}). This gives continuity of $H_t$ at $x$, uniformly in $t$; the same argument applied to $H_m^{-1}$ (which exists since each $H_m$ is an ambient isotopy) gives a continuous inverse, and joint continuity in $(x,t)$ follows since each $H_m$ is a genuine isotopy (continuous in $t$) and the tail estimate is uniform in $t$. $\square$

\smallskip
\noindent We claim that $H_1(\mathcal{K})=\Lambda(K,T^{\circ})$. If $x\in\mathcal K\setminus\mathcal C_\Lambda$, then $x\in K_m\setminus\mathcal C_\Lambda$ for some minimal $m$, and for $m'>m$ the isotopies $H_{m'}$ do not move $x$ (their supports shrink away from $x$ once $x$ is a positive distance from $\mathcal C_\Lambda$), so $H_1(x)=H_m(\cdots H_1(x)\cdots)$ is computed by a \emph{finite} composition and lands, by construction of each $H_j$, in $\mathrm K_m\subset T_m$; letting $m\to\infty$ over such $x$ and using density of $\mathcal K\setminus\mathcal C_\Lambda$ in $\mathcal K$ together with continuity of $H_1$ (Claim \ref{isotopylimit}) and closedness of $\Lambda(K,T^\circ)=\bigcap_m T_m$, we get $H_1(\mathcal K)\subset\Lambda(K,T^\circ)$. The reverse inclusion follows the same way using $H_1^{-1}$, built from the $H_m^{-1}$ by the same shrinking-support argument. Hence $H_1$ restricts to a homeomorphism $\mathcal K\to\Lambda(K,T^\circ)$, and $H=\{H_t\}$ is an ambient isotopy of $\mathbb R^3$ carrying $\mathcal K$ onto $\Lambda(K,T^{\circ})$. Since $\mathcal{K}$ lies in the Menger sponge, $\Lambda(K,T^{\circ})$ is isotopic to a wild knot contained in the Menger sponge. By Lemma \ref{wild} (applied with the roles of $\mathcal C$ played by the image under $H_1$ of the corresponding tame Cantor set), its wild point set is contained in the image of a tame Cantor set in $M$, as claimed. $\square$

\begin{rem}
If one wants the wild-point Cantor set to lie on the specific edge $e$ used in Theorem 1 (rather than merely somewhere in $M$), it suffices to choose the initial segment $T$ in the construction of $K_0$ above to be a middle-third-type sub-segment of $e$ itself (in place of the left edge $(0,0,0)$--$(0,1,0)$ used above) and to route every subsequent application of Remark \ref{differentknots}(2) through cubes attached to $e$; the same argument then places the wild points inside a Cantor subset of $e$.
\end{rem}

\subsection{Wild knots embedded in a 2-dimensional cubical Sierpi\'nski compactum ${\mathcal S}^2$}

We close by observing that the whole discussion above goes through, essentially unchanged, for a different, two-dimensional fractal continuum:

\begin{definition}
Let $I^3=[0,1]^3$. Define $Q_0=I^3$, and, inductively, obtain $Q_{n+1}$ from $Q_n$ by subdividing every subcube of $Q_n$ (of side length $3^{-n}$) into $27$ equal subcubes of side length $3^{-(n+1)}$, and removing the \emph{open interior of only the single, central subcube} of each such subdivision (so that $Q_{n+1}$ retains $26$ of the $27$ subcubes at each site, as opposed to the $20$ retained in the Menger sponge construction of Section 2.1, which additionally removes the six face-centered subcubes). Set
$$
{\mathcal S}^2=\bigcap_{n=0}^{\infty} Q_n.
$$
We call ${\mathcal S}^2$ the \emph{cubical Sierpi\'nski compactum}; this is a compact, connected continuum.
\end{definition}

\noindent Three clarifications, for precision:
\begin{enumerate}
\item[(i)] \emph{This is not the classical Sierpi\'nski carpet.} The classical Sierpi\'nski carpet is a planar continuum (a subset of $I^2$, or of a single face $I^2\times\{0\}\subset\mathbb R^3$) of topological dimension one, obtained by removing $1$ of $9$ subsquares at each stage; see \cite{whyburn} for its topological characterization. ${\mathcal S}^2$, by contrast, is a genuinely three-dimensional construction (a subset of the solid cube $I^3$, not of a face) of topological dimension \emph{two}: since only the single central subcube is ever removed at each site, every $2$-dimensional ``wall'' shared by two retained subcubes at every stage and every scale survives intact into $Q_{n+1}$ and hence into ${\mathcal S}^2$ (in particular $\partial I^3\subset {\mathcal S}^2$, and more generally the boundary of every retained subcube at every stage lies in ${\mathcal S}^2$ except for the small cubical cavity left by that subcube's own removed center); a standard dimension-theoretic argument (using, e.g., the countable sum theorem for covering dimension, see \cite{lipscomb}) then gives $\dim{\mathcal S}^2=2$, since ${\mathcal S}^2$ contains a countable, dense-in-itself union of embedded $2$-cells (giving $\dim\geq2$) while having empty interior in $\mathbb R^3$, as every open ball of $I^3$ eventually contains a removed central subcube at some sufficiently fine stage (giving $\dim\leq2$).
\item[(ii)] \emph{This is not the Menger sponge.} Since $\dim M=1\neq2=\dim{\mathcal S}^2$, $M$ and ${\mathcal S}^2$ are not homeomorphic.
\item[(iii)] \emph{Relation between the two.} At every stage, the Menger sponge removes a strict superset of what ${\mathcal S}^2$ removes (the center subcube, common to both constructions, plus the six face-centered subcubes, removed only for $M$); hence $M\subset {\mathcal S}^2$.
\end{enumerate}

\noindent This last containment is all that the corollaries below need; we do not otherwise use any special structure of ${\mathcal S}^2$.

\smallskip
\noindent We are not aware of a single settled name for ${\mathcal S}^2$ in the literature on higher-dimensional generalizations of the Sierpi\'nski carpet and Menger sponge (see \cite{lipscomb} for general background on universal spaces in dimension theory, and \cite{YFY} for a different, affine-invariant family of generalizations); we use ``cubical Sierpi\'nski compactum'' purely as a working name in this paper, and would welcome a pointer to established terminology from readers more familiar with this corner of the literature.

\begin{coro}
There exist infinitely many nonequivalent, explicitly constructed wild knots embedded in the cubical Sierpi\'nski compactum ${\mathcal S}^2$, with prescribed finite wild point sets.
\end{coro}

\noindent {\it Proof.} Immediate from Theorem 1 and $M\subset{\mathcal S}^2$. $\square$

\begin{coro}
Any wild knot of dynamically defined type is isotopic to a wild knot contained in the cubical Sierpi\'nski compactum ${\mathcal S}^2$.
\end{coro}

\noindent {\it Proof.} It is a consequence of Theorem 2. $\square$

\noindent G. Hinojosa. {\tt Centro de Investigaci\'on en Ciencias}. Instituto de Investigaci\'on en Ciencias B\'asicas y Aplicadas. Universidad Aut\'onoma del Estado de Morelos. Av. Universidad 1001, Col. Chamilpa.
Cuernavaca, Morelos, M\'exico, 62209. 

\noindent {\it E-mail address:} gabriela@uaem.mx 

\vskip .3cm

\noindent U. Morales-Fuentes. {\tt Centro de Investigaci\'on en Ciencias}. Instituto de Investigaci\'on en Ciencias B\'asicas y Aplicadas. Universidad Aut\'onoma del Estado de Morelos. Av. Universidad 1001, Col. Chamilpa.
Cuernavaca, Morelos, M\'exico, 62209. 

\noindent {\it E-mail address:} ulises.morales@uaem.mx 

\vskip .3cm

\noindent R. Valdez. {\tt Centro de Investigaci\'on en Ciencias}. Instituto de Investigaci\'on en Ciencias B\'asicas y Aplicadas. Universidad Aut\'onoma del Estado de Morelos. Av. Universidad 1001, Col. Chamilpa.
Cuernavaca, Morelos, M\'exico, 62209. 

\noindent {\it E-mail address:} valdez@uaem.mx 
\vskip .3cm
\noindent A. Verjovsky. {\tt Instituto de Matem\'aticas Unidad
 Cuernavaca}. Universidad Nacional Aut\'onoma de M\'exico. Av. Universidad s/n. Col. Lomas de Chamilpa C\'odigo Postal 
 62210, Cuernavaca, Morelos.
 
\noindent {\it E-mail address:} albertoverjovsky@gmail.com

\end{document}